\documentclass[11pt]{article}
\usepackage{syeon}

\title{Hydrodynamic and Dynamical $\Phi^4$ Limits of Ising–Kac Models with Diffusion-Based Interactions on the Sierpi\'nski gakset}

\author{Hongyi Chen\thanks{Department of Mathematics, Aarhus University, Aarhus, Denmark. Email: \texttt{hchen77@math.au.dk}} \and Seoyeon Yang \thanks{Department of Mathematics, Princeton University, Princeton, U.S.A. Email: \texttt{syeon.y@princeton.edu}}}

\date{\today}

\begin{document}

\maketitle

\begin{abstract}
\noindent 
We introduce a general form of Ising--Kac interactions through
heat-kernel averaging, with the heat-time parameter replacing the role of 
the interaction range.
From the viewpoint of Dirichlet forms theory,
this construction provides an intrinsic analogue of the mesoscopic averaging in classical Kac models when translation invariance is unavailable.
We then investigate limits of this model on graph approximations of the Sierpi\'nski gasket.
For the associated Glauber dynamics, we establish a hierarchy of scaling limits for the magnetization field. 
First, at fixed interaction heat time and inverse temperature,
we derive a nonlocal hydrodynamic equation on the continuum gasket.
Second, as the heat time vanishes and the inverse temperature
approaches its mean-field critical value, suitably rescaled
hydrodynamic solutions converge to the solution of the deterministic
Allen--Cahn equation.
Finally, under a joint scaling of graph refinement, interaction
heat time, and inverse temperature near criticality,
we prove convergence of critical magnetization fluctuations to the solution of the dynamical $\Phi^4$ equation.
These results provide a concrete realization of the Ising-$\Phi^4$ correspondence in a geometric setting with non-integer spectral dimension.


\end{abstract}

\tableofcontents


\section{Introduction}\label{sec: Intro}

{One realization of the} universality {principle is that} microscopic spin systems can exhibit the same {rescaled} continuum critical dynamics.
For Ising-type systems whose dynamics do not conserve magnetization, the dynamical $\Phi^4$ equation provides {one such} basic continuum model~\cite{HohenbergHalperin1977}.
Rigorous microscopic realizations are given by suitably rescaled Ising--Kac Glauber dynamics on Euclidean lattices~\cite{IsingKac1D,isingkacLine,IsingKac-Phi4-2D,3Dcase}.
In these models, the Kac interaction averages spins over mesoscopic neighborhoods through a translation-invariant convolution kernel. 
{In the absence of translation structure on the gasket, we define Kac averaging from a prescribed diffusion, and establish the corresponding scaling limits.}

{One motivation for pursuing such non-integer dimension realizations of Ising-$\Phi^4$ universality} comes from Wilson's renormalization-group approach to critical phenomena and the role of noninteger dimension in the Wilson--Fisher expansion~\cite{WilsonFisher1972,WilsonKogut1974}.
Fractals provide concrete geometries for investigating critical behavior, but are not simply realizations of continued Euclidean theory.
Analytical and numerical studies of Ising models on fractal lattices indicate that equilibrium critical behavior can depend on geometric and connectivity properties not captured by fractal dimension alone~\cite{GefenMandelbrotAharony1980,MonceauHsiao2004}.
These observations suggest using a chosen diffusion as the starting point for defining microscopic interactions and studying their continuum limits, rather than relying on dimension alone.

We propose to use the heat semigroup of a prescribed diffusion
as the interaction average.
{
At the level of the interaction, the construction applies to general graphs equipped with a symmetric graph
Laplacian. The scaling limits established here concern graph
approximations of the Sierpi\'nski gasket.}
On graph approximations \(G_N\) of the Sierpi\'nski gasket \(K\),
we implement this construction through the discrete heat semigroup
\(P_\eta^N=e^{\eta L_N}\), with self-interactions removed.
Here \(L_N\) is the {rescaled graph Laplacian}, and the interaction heat time
\(\eta\) controls the scale of averaging.
Thus the interaction and the continuum diffusion are tied to the same underlying Dirichlet form, rather than to a convolution structure.
{Details of the construction can be found in} Section~\ref{sec:Heat Ker Kac}.

For the associated Glauber dynamics, we establish a hierarchy of
hydrodynamic and critical-temperature scaling limits on the gasket.
At fixed interaction heat time and inverse temperature, we derive
a nonlocal hydrodynamic equation
(Theorem~\ref{thm:hydro-lim}).
We then let the heat time vanish while tuning the inverse temperature
near its mean-field critical value and rescaling time and
magnetization.
The rescaled hydrodynamic solutions converge to the solution of
the deterministic Allen--Cahn equation
(Theorem~\ref{thm:allen-cahn-limit}).
These two results also give a successive-limit statement for the
microscopic dynamics (Corollary~\ref{cor:crit temp hydrp}),
in which the microscopic fluctuations disappear.

Our principal result, Theorem~\ref{thm:stochastic-limit}, shows that
this formulation realizes the dynamical Ising--Kac-to-$\Phi^4$
correspondence on $K$.
Under a joint scaling of the graph level, the interaction heat time,
and the inverse temperature near criticality,
the rescaled magnetization field converge to the stochastic Allen--Cahn equation
\[
    \partial_t\Phi
    =
    L\Phi+A\Phi-\frac13\Phi^3+\sqrt{2/\theta}\,\xi.
\]
Here \(L\) is the Neumann Laplacian on \(K\), and \(\xi\) is
space--time white noise with intensity \(dt\,m(dx)\), where \(m\)
is the normalized self-similar measure.
The constants \(A\in\mathbb R\) and \(\theta\in(0,\infty)\)
are determined by the temperature tuning and fluctuation scaling.
Since the spectral dimension of $K$ is $d_s=2\log3/\log5<2$, the
equation requires no renormalization (Remark~\ref{rmk:higher dim}).
The contribution is a diffusion-based formulation of Kac interactions,
together with a rigorous realization of its hydrodynamic and
nonlinear stochastic limits on the gasket.

\paragraph{Organization of the paper.}
The remainder of this section defines the model, states the main
results, and discusses their relation to previous work.
Section~\ref{sec:gasket} recalls the resistance geometry,
harmonic extension, and heat-semigroup approximation on the gasket.
Sections~\ref{sec:hydro proof} and~\ref{sec:det AC lim proof}
prove Theorems~\ref{thm:hydro-lim}
and~\ref{thm:allen-cahn-limit}, respectively.
Section~\ref{sec:phi4 lim proof} establishes global well-posedness
of the limiting SPDE and proves
Theorem~\ref{thm:stochastic-limit}, following the local-field
strategy of Mourrat and Weber~\cite{IsingKac-Phi4-2D}.

\subsection*{Acknowledgments}
We thank Michael Aizenman, Patricia Goncalves, Martin Hairer, Konstantin Matetski, and Hendrik Weber for helpful discussions and encouragement. H.C. is supported by the research grant (VIL73729) from Villum Fonden. 

\subsection*{Declaration of generative AI use}
We gratefully acknowledge the assistance of GPT-6 Astra plus for detail refinement (including the identification and removal of logical gaps) and presentation optimization. The authors reviewed and edited all AI-assisted material and take full responsibility for the mathematical statements, proofs, and presentation of the article. 

\subsection{Heat Kernel type Kac interaction}\label{sec:Heat Ker Kac}

Let \(K\) be the standard Sierpi\'nski gasket, equipped with its normalized
self-similar measure \(m\). We write \(L\) for the Neumann Laplacian, with the
convention that \(L\) is non-positive, and
\[
P_t=e^{tL},\qquad t\geq0,
\]
for the associated conservative symmetric heat semigroup. Its kernel \(p_t(x,y)\) is a density with respect to \(m\):
\[
(P_tf)(x)=\int_K p_t(x,y)f(y)\,m(dy).
\]
The generator domain on \(C(K)\) is
\[
D_C(L)=\left\{f\in C(K):\lim_{t\downarrow0}\frac{P_tf-f}{t}
\text{ exists in the uniform norm}\right\},
\]
and \(Lf\) denotes this limit. The construction and the discrete approximation of these objects are recalled in Section~\ref{sec:gasket}.

The feature of Kac interactions that we seek to retain is
\emph{mesoscopic averaging}.
We choose to replace translation-based Kac averaging by intrinsic
diffusion-scale averaging:

\begin{equation}\label{eq:intro-main-principle}
    \qquad
    J_\eta * u
    \quad\longrightarrow\quad
    P_\eta u .
    \qquad
\end{equation}

On Euclidean space, $P_\eta$ is convolution with the Gaussian heat kernel,
so \eqref{eq:intro-main-principle} selects a distinguished Kac family.
On a space without translation invariance, the heat semigroup remains
intrinsic: it is determined by the underlying Dirichlet form.
It also connects the nonlocal interaction to the limiting diffusion.
For functions in the generator domain,
\begin{equation}
\frac{P_\eta-I}{\eta}f\longrightarrow Lf
\quad\text{in }C(K),\qquad f\in D_C(L),\quad\eta\downarrow0.
\label{eq:revised-intro-2}
\end{equation}
{This choice also has an energy interpretation. Let \(\mathcal E\) denote
the Dirichlet energy recalled in Section~\ref{sec:gasket}. For
\(f\in D_C(L)\), symmetry and conservativity give
\[
\frac1{2\eta}\iint_{K\times K}
p_\eta(x,y)\bigl(f(x)-f(y)\bigr)^2\,m(dx)m(dy)
=\frac{\langle f,(I-P_\eta)f\rangle_{L^2(m)}}{\eta}
\longrightarrow\mathcal E(f,f),\qquad\eta\downarrow0.
\]
Thus the nonlocal quadratic energy defined by the interaction kernel approximates the energy of the prescribed diffusion.}
The parameter \(\eta\) is an interaction heat time, not the evolution time of the spin system.
Thus the same operator plays two roles.  At positive heat time it supplies the nonlocal Kac average, while at vanishing heat time its infinitesimal generator produces the local diffusion operator.  
This averaging/generator duality is the main structural reason for using the heat semigroup.

For $N\geq 0$, let $G_N=(V_N,E_N)$ be the level-$N$ graph approximation, where $x\sim y$ means that $x,y\in V_N$ are joined by an edge. We write
\[
  m_N=\frac{1}{|V_N|}\sum_{x\in V_N}\delta_x
\]
for the uniform probability measure on $V_N$.
We use the graph Laplacian
\begin{equation}
    (L_Nf)(x)
        :=
        \frac32\,5^N
        \sum_{\substack{y\in V_N\\y\sim x}}
        \bigl(f(y)-f(x)\bigr),
    \qquad x\in V_N,
    \label{eq:discrete-L}
\end{equation}
and denote by
\[
        P_t^N=e^{tL_N}
\]
its heat semigroup.
The factor $5^N$ is the natural diffusion-time renormalization on the gasket, while $3/2$ matches the conventional normalization of the continuum Neumann Laplacian.
Let $p_t^N(x,y)$ denote the heat-kernel density with respect to $m_N$, so that
\begin{equation}
    (P_t^Nf)(x)
      =
      \frac1{|V_N|}
      \sum_{y\in V_N}p_t^N(x,y)f(y).
    \label{eq:discrete-heat-kernel}
\end{equation}
Equivalently, $|V_N|^{-1}p_t^N(x,y)$ is the transition probability of the rescaled reflecting random walk from $x$ to $y$ in time $t$.
These densities converge uniformly to $p_t$ on compact time intervals contained in $(0,\infty)$, as established in Proposition~\ref{prop:gasket-approximation}.

For $\eta>0$, we define the heat-kernel Kac field of a spin configuration
\[
        \sigma\in\Sigma_N:=\{-1,+1\}^{V_N}
\]
by
\begin{equation}
    h_{N,\eta}(\sigma,x)
       :=
       \frac1{|V_N|}
       \sum_{\substack{y\in V_N\\y\neq x}}
       p_\eta^N(x,y)\sigma(y)
       =(P_\eta^N\sigma)(x)
      -
      \frac1{|V_N|}p_\eta^N(x,x)\sigma(x)
       .
    \label{eq:heat-kernel-Kac-field}
\end{equation}
The diagonal is removed so that a spin does not contribute to its own local field. The omitted
diagonal weight vanishes as $N\to\infty$ for fixed $\eta$; its effect
under the joint critical scaling is estimated in
Section~\ref{sec:phi4 lim proof}.
{The interaction weights are symmetric and nonnegative, so they define
a ferromagnetic Ising model with Hamiltonian
\[
\mathcal H_{N,\eta}(\sigma)
=-\frac12\sum_{x\in V_N}\sigma(x)h_{N,\eta}(\sigma,x).
\]
Given inverse temperature $\beta\geq0$, the heat-bath Glauber rate
for the Gibbs measure proportional to \(e^{-\beta\mathcal H_{N,\eta}}\)
at $x\in V_N$ is}
\begin{equation}
    c_{N,\eta,\beta}(\sigma,x)
       :=
       \frac12
       \left[
          1-\sigma(x)
          \tanh\!\bigl(\beta h_{N,\eta}(\sigma,x)\bigr)
       \right].
    \label{eq:glauber-rate}
\end{equation}
These rates favor alignment with the local field, with $\beta$ controlling the strength of this preference.
If $\sigma^x$ denotes the configuration obtained by flipping the spin at $x$, the Markov generator acting on $F:\Sigma_N\to\mathbb R$ is
\begin{equation}
    (\mathcal G_{N,\eta,\beta}F)(\sigma)
       =
       \sum_{x\in V_N}
       c_{N,\eta,\beta}(\sigma,x)
       \bigl[F(\sigma^x)-F(\sigma)\bigr].
    \label{eq:spin-generator}
\end{equation}
For the process \((\sigma_t^N)_{t\geq0}\) generated by \eqref{eq:spin-generator}, define the empirical magnetization measure
\begin{equation}
    \pi_t^N
       :=
       \frac1{|V_N|}
       \sum_{x\in V_N}\sigma_t^N(x)\delta_x.
    \label{eq:empirical-magnetization}
\end{equation}
For $\varphi\in C(K)$, define
\[
    \langle\pi_t^N,\varphi\rangle
      =
      \frac1{|V_N|}
      \sum_{x\in V_N}\varphi(x)\sigma_t^N(x).
\]
When a discrete function is viewed as a function on \(K\), we use its
energy-minimizing harmonic extension, denoted by \(\iota_N\) and recalled in Section~\ref{sec:gasket}.

\subsection{Main results}
We first state two
deterministic limits that identify the drift and explain its scaling: the nonlocal
hydrodynamic limit at fixed interaction heat time, followed by the critical
Allen--Cahn limit. Their combination gives a successive-limit statement for the spin
system. Theorem~\ref{thm:stochastic-limit} instead takes the graph and interaction
scales to their limits simultaneously, retaining a nondegenerate stochastic term.
\subsubsection{Fixed heat time: the nonlocal hydrodynamic equation}\label{sec: hydro result}

For fixed $\eta>0$, the interaction averages over a macroscopic region
containing an increasing number of vertices as $N\to\infty$.
A flip changes $\sigma(x)$ by $-2\sigma(x)$, so
\eqref{eq:glauber-rate} gives the local drift
$-\sigma(x)+\tanh(\beta h_{N,\eta}(\sigma,x))$.
In the hydrodynamic limit, heat-kernel averaging replaces the local
field by $P_\eta u$, while fluctuations of the empirical observables
vanish. This leads to the nonlocal equation
\begin{equation}\label{eq:nonlocal-hydrodynamic}
  \partial_tu_t(x)=-u_t(x)+\tanh(\beta P_\eta u_t(x)),\qquad u(0,x)=u_0(x).
\end{equation}

\begin{theorem}[Hydrodynamic limit]
\label{thm:hydro-lim}
Fix $\eta>0$, $\beta\geq0$, and $u_0\in C(K)$. Equation
\eqref{eq:nonlocal-hydrodynamic} has a unique global solution
$u\in C^1([0,\infty);C(K))$.
Suppose that the Glauber dynamics with generator
$\mathcal G_{N,\eta,\beta}$ have initial laws $\mu_N$ satisfying
\[
\langle\pi_0^N,\varphi\rangle
\xrightarrow[N\to\infty]{\mathbb P}
\int_K\varphi(x)u_0(x)\,m(dx)
\qquad\text{for every }\varphi\in C(K).
\]
Then, for every $T>0$ and $\varphi\in C(K)$,
\[
\sup_{0\le t\le T}
\left|
\left\langle\pi_t^N,\varphi\right\rangle
-
\int_K\varphi(x)u_t(x)\,m(dx)
\right|
\xrightarrow[N\to\infty]{\mathbb P}
0.
\]
\end{theorem}


\subsubsection{Critical scaling: the deterministic Allen--Cahn equation}\label{sec:hydro to AC}


The critical scaling is suggested by spatially constant solutions of
\eqref{eq:nonlocal-hydrodynamic}.
Since \(P_\eta\) preserves constants, these satisfy
\[
    a'(t)=-a(t)+\tanh(\beta a(t)).
\]
Near \(a=0\),
\[
    -a+\tanh(\beta a)
    =
    (\beta-1)a-\frac{\beta^3}{3}a^3+O(a^5).
\]
Thus the homogeneous linear term vanishes at the mean-field
critical value \(\beta=1\).

For spatially varying profiles, the linearized drift at this value
is instead \((P_\eta-I)u\).
By \eqref{eq:revised-intro-2},
\((P_\eta-I)f/\eta\to Lf\) for \(f\in D_C(L)\), suggesting
time acceleration by \(\eta^{-1}\) to retain the diffusion term.
On this time scale, an inverse-temperature deviation of order
\(\eta\) and magnetization of order \(\sqrt\eta\) balance the
linear and cubic reaction terms.
These observations lead to the critical rescaling in the
following theorem.

\begin{theorem}[Deterministic Allen--Cahn limit]
\label{thm:allen-cahn-limit}
Fix $A\in\mathbb R$ and $\beta_\eta\ge0$ such that
$(\beta_\eta-1)/\eta\to A$. 
Fix an initial profile $X_0 \in C(K)$, and let $u^\eta$ be the solution to \eqref{eq:nonlocal-hydrodynamic} with initial data $u^\eta(0,\cdot) = \sqrt{\eta}\,X_0$.
For $t \geq 0$, define the  rescaled magnetization field
\begin{equation}
    X_t^\eta := \eta^{-1/2} u^\eta_{t/\eta}.
\end{equation}
Then, for every $T>0$, 
\[
    X^\eta \longrightarrow X \qquad \text{in } \ C([0,T];C(K)) \quad \text{as } \eta\downarrow0,
\]
where $X$ is the unique mild solution to the local Allen--Cahn equation
\begin{equation}
    \begin{cases}
    \partial_t X = LX + AX - \frac{1}{3}X^3, & t>0,\\[0.3em]
    X(0,\cdot) = X_0.
    \end{cases}
    \label{eq:allen-cahn}
\end{equation}
\end{theorem}

Combining this theorem with the hydrodynamic limit gives a consequence
for the microscopic dynamics. Write $\sigma^{N,\eta}$ for the process
with generator $\mathcal G_{N,\eta,\beta_\eta}$ and $\pi^{N,\eta}$ for
its empirical magnetization. Its rescaled empirical field is
\begin{equation}\label{eq:microscopic-critical-field}
 \mathcal X_t^{N,\eta}(\varphi)
 =\eta^{-1/2}\langle\pi_{t/\eta}^{N,\eta},\varphi\rangle
 =\frac1{\sqrt\eta\,|V_N|}
   \sum_{x\in V_N}\varphi(x)\sigma_{t/\eta}^{N,\eta}(x),
 \qquad \varphi\in C(K).
\end{equation}
Suppose that, for each sufficiently small fixed
$\eta>0$, the initial laws satisfy
\[
\mathcal X_0^{N,\eta}(\psi)
\xrightarrow[N\to\infty]{\mathbb P}
\int_K X_0(x)\psi(x)\,m(dx)
\qquad\text{for every }\psi\in C(K).
\]
so the initial spin laws are associated with the profile
$\sqrt\eta\,X_0$. Then we have the following Corollary.

\begin{corollary}\label{cor:crit temp hydrp}
Under the temperature scaling of Theorem~\ref{thm:allen-cahn-limit}, fix \(X_0\in C(K)\).
Then, for every \(T,\varepsilon>0\) and \(\varphi\in C(K)\),
\[
 \lim_{\eta\downarrow0}\limsup_{N\to\infty}
 \mathbb P\!\left(
  \sup_{0\le t\le T}
  \left|\mathcal X_t^{N,\eta}(\varphi)
        -\int_KX(t,x)\varphi(x)\,m(dx)\right|>\varepsilon
 \right)=0,
\]
where $X(t,\cdot)$ is the solution of determinisitc Allen-Cahn \eqref{eq:allen-cahn} starting from \(X_0\in C(K)\).
\end{corollary}

\begin{proof}
For fixed $\eta$, Theorem~\ref{thm:hydro-lim} on $[0,T/\eta]$
gives convergence in probability of $\mathcal X^{N,\eta}(\varphi)$
to $\int_KX^\eta(t)\varphi\,dm$, uniformly for $t\le T$.
The conclusion follows by the triangle inequality and
Theorem~\ref{thm:allen-cahn-limit}.
\end{proof}

\subsubsection{Joint critical scaling: the dynamical \(\Phi^4\)/stochastic Allen--Cahn equation}\label{sec:phi4 lim}

The successive limits in Corollary~\ref{cor:crit temp hydrp} eliminate
the microscopic fluctuations before the interaction heat time vanishes.
To capture the critical stochastic fluctuations, we apply the same time and amplitude scalings but
choose \(\eta=\eta_N\downarrow0\) jointly with
\(N\to\infty\) to make microscopic
fluctuations remain comparable to the critical magnetization scale.

A flip at \(x\) changes the field in \eqref{eq:microscopic-critical-field} by
\[
-\frac{2\varphi(x)\sigma(x)}{\sqrt{\eta_N}\,|V_N|},
\]
and the time acceleration multiplies each flip rate by \(\eta_N^{-1}\). The natural
scale of the predictable quadratic-variation rate is therefore
\((|V_N|\eta_N^2)^{-1}\). We impose the balance
\begin{equation}
|V_N|\eta_N^2\longrightarrow\theta\in(0,\infty).
\label{eq:noise-balance}
\end{equation}
Since \(|V_N|=(3^{N+1}+3)/2\), this gives \(\eta_N\asymp3^{-N/2}\) and, in particular,
\(5^{-N}\ll\eta_N\ll1\). Thus the interaction is mesoscopic: its heat time vanishes on
the continuum scale but remains much larger than the time scale for a single graph
step.

When the rescaled local field \(\eta_N^{-1/2}h_{N,\eta_N}\) remains bounded, the flip
rates converge uniformly to \(1/2\). yielding the noise coefficient $\sqrt{2/\theta}$. This approximation
is justified by localization in
Section~\ref{sec:phi4 lim proof}.
With the same temperature tuning as in
Theorem~\ref{thm:allen-cahn-limit}, the candidate limiting equation is
\begin{equation}\label{eq:phi4}
 \begin{cases}
  \partial_t\Phi=L\Phi+A\Phi-\frac13\Phi^3
                  +\sqrt{2/\theta}\,\xi, & t>0,\\
  \Phi(0,\cdot)=\Phi_0,
 \end{cases}
\end{equation}
where \(\Phi_0\in C(K)\) is given and \(\xi\) is space--time white noise with
intensity \(dt\,m(dx)\).
For small $t>0$, the on-diagonal heat-kernel bound $p_t(x,x)\le Ct^{-\frac{d_s}{2}}$ holds with $\frac{d_s}{2}=\frac{\log 3}{\log 5}< 1$, so $\int_0^t p_{2s}(x,x)\,ds<\infty$ and the stochastic convolution has finite pointwise variance. The bounds are discussed with more detail in Section~\ref{sec:gasket}; continuity of the stochastic convolution
is recalled and global well-posedness of \eqref{eq:phi4} is proved at
the start of Section~\ref{sec:phi4 lim proof}.

We state the initial condition in terms of the heat-smoothed empirical
field. Discrete functions in this condition are harmonically extended
to $K$, as recalled in Section~\ref{sec:gasket}.
The path space in the following theorem carries the local $J_1$
topology.

\begin{theorem}[Dynamical $\Phi^4$ on the gasket]\label{thm:stochastic-limit}
Assume the fluctuation balance \eqref{eq:noise-balance} and
$(\beta_{\eta_N}-1)/\eta_N\to A$. If
\[
 \iota_N\!\left[
   \mathcal X_0^{N,\eta_N}(p_{\eta_N}^N(x,\cdot))\right]
 \Longrightarrow\Phi_0
 \quad\text{in }C(K),
\]
then $\mathcal{X}^{N,\eta}$ converges in law to the solution $\Phi$ of the stochastic Allen--Cahn equation \eqref{eq:phi4}: for every finite collection
$\varphi_1,\ldots,\varphi_k\in D_C(L)$,
\[
 \bigl(\mathcal X^{N,\eta_N}(\varphi_i)\bigr)_{i=1}^k
 \Longrightarrow
 \left(\int_K\Phi_t(x)\varphi_i(x)\,m(dx)\right)_{i=1}^k
 \quad\text{in }D(\mathbb R_+;\mathbb R^k).
\]
\end{theorem}

\begin{remark}\label{rmk:higher dim}
    The spectral dimension governs how effectively this diffusion smooths
space--time white noise. Write $p_t(x,y)$ for the heat-kernel density
with respect to the normalized self-similar measure $m$ on $K$.
As mentioned earlier, the short-time heat kernel diagonal estimate on the gasket is
$p_t(x,x)\asymp t^{-d_s/2}$, where $d_s=2\log3/\log5<2$
\cite{BarlowPerkins1988,Kigami2001}.
By symmetry, and the semigroup identity, the stochastic
convolution driven by unit space--time white noise has pointwise variance
\[
 \int_0^t\int_K p_s(x,y)^2\,m(dy)\,ds
 =\int_0^t p_{2s}(x,x)\,ds.
\]
The integral is finite because $d_s<2$. Heat-kernel regularity estimates
further give a continuous version of the stochastic convolution
\cite{HamblyYang2018}, so the cubic term in the limiting equation is
defined pointwise without renormalization.
When \(d_s\ge2\), the linear field is already too singular for this pointwise interpretation. Consequently, the ordinary cubic must be replaced by a renormalized expression.\\
For extensions to spaces with spectral dimension $d_s\ge4$, formal power
counting predicts criticality at $d_s=4$ and supercriticality at $d_s>4$.
In Euclidean space, the cubic stochastic Allen--Cahn equation is locally
subcritical precisely when $d<4$, so the subcritical renormalization
framework does not directly extend to this regime
\cite[Section~8]{Hairer2014}.
Gaussianity results for critical Euclidean Ising and $\varphi^4$ scaling
limits suggest looking for Gaussian limits
\cite[Sections~8 and~10]{Aizenman2022}.
These equilibrium results, however, do not determine the limiting
dynamics on fractals. Identifying the appropriate scaling and whether
the resulting limit is Gaussian or retains a nonlinear interaction
requires a separate analysis.
\end{remark}

{
The proof of Theorem~\ref{thm:stochastic-limit} follows the
local-field strategy of Mourrat and
Weber~\cite{IsingKac-Phi4-2D}.
We work with the interaction-averaged field, compare truncated
discrete and continuum mild equations, and remove the truncation
using global continuity of the limiting solution.
Heat-kernel and resistance estimates give uniform spatial and
temporal control of the stochastic convolutions along the
joint scaling.
These estimates supplement strong semigroup convergence and
retain the diagonal deletion, whose effect must be controlled
after time acceleration.}

\subsection{Relation to previous work}
\label{relation-to-previous-work}

{Euclidean Ising--Kac limits for systems undergoing Glauber dynamics have had a long history.}
De Masi, Orlandi, Presutti, and Triolo~\cite{Kac_Glauber_Hydrodynamics_1994} derived a nonlocal mesoscopic evolution equation from Glauber--Kac dynamics; Theorem~\ref{thm:hydro-lim} is its heat-semigroup counterpart on the gasket.
Critical stochastic limits were established in one dimension by Bertini, Presutti, R\"udiger, and Saada~\cite{IsingKac1D} and by Fritz and R\"udiger~\cite{isingkacLine}, in two dimensions by Mourrat and Weber~\cite{IsingKac-Phi4-2D}, and in three dimensions by Grazieschi, Matetski, and Weber~\cite{3Dcase}.
In two and three dimensions, the limiting field is distribution-valued, and renormalization of the nonlinearity is reflected in the microscopic inverse-temperature tuning.
In terms of the regularity of the limiting field, our result is closer to the one-dimensional case.
The emphasis here is on defining Kac averaging through a prescribed diffusion and establishing the resulting hydrodynamic and critical stochastic limits on the gasket.

The absence of translation invariance does not by itself prevent the construction of local kernels that approximate a diffusion.
On smooth compact Riemannian manifolds, van Ginkel and Redig~\cite{vGR20} construct weighted graph approximations and prove a hydrodynamic limit for symmetric exclusion.
Their generator approximation uses kernels depending on geodesic distance and Taylor expansions in smooth coordinates to recover a constant multiple of the Laplace--Beltrami operator; see~\cite[Section~3.3, Proposition~3.12]{vGR20}.
{This was generalized to non-compact manifolds in \cite{JRR26+}.}
Our construction instead starts from a chosen diffusion and uses its heat semigroup as the averaging operator.
The relation between this average and the continuum generator is then supplied by the semigroup definition, as expressed in \eqref{eq:revised-intro-2}, rather than by a spatial Taylor expansion.
This provides a way to define Kac interactions without requiring a smooth-coordinate structure, while leaving the convergence of the corresponding microscopic dynamics to be established.

{Investigation of hydrodynamic limits and their fluctuations for particle systems on fractal spaces appears to be a much more recent phenomenon.} Jara~\cite{Jara09} derived a nonlinear diffusion equation from a zero-range process and, in~\cite{Jara2006}, obtained a generalized Ornstein--Uhlenbeck equilibrium fluctuation limit.
Chen and Gon\c{c}alves~\cite{CG-hydro} established hydrodynamic and equilibrium Ornstein--Uhlenbeck fluctuation limits for boundary-driven exclusion.
Van Meurs and Tsunoda~\cite{GlauberKawasaki} obtained a deterministic reaction--diffusion equation from Glauber--Kawasaki dynamics.
Our critical scaling instead retains a cubic drift together with nondegenerate stochastic forcing.
There is no Kawasaki exchange in our model: the diffusion term in the limit arises from the interaction average in the Glauber rates.
In dimension 1, a recent preprint by Zhang \cite{QZhang26} claims to have identified the stochastic Cahn-Hillard equation as the continuum fluctuation limit for the Ising-Kac model undergoing Kawasaki dynamics.

The Ising--$\Phi^4$ correspondence also has classical equilibrium formulations~\cite{GriffithsSimon1973,Aizenman1982}. 
In dimension 2, Hairer and Iberti \cite{HI18} showed, building on \cite{IsingKac-Phi4-2D}, that the Ising-Kac Gibbs measures scale to the $\Phi^4_2$ measure on the 2-torus.
Field-theoretic approaches on fractals include models with heat-kernel and hierarchical\ propagators~\cite{Eyink1989QuantumFM} and perturbative quantum field theory on random trees~\cite{QFTonCRT}.
These works provide a broader context for studying spin systems and continuum fields outside Euclidean geometry, but address different questions from the microscopic dynamical limit considered here.

{Continuum SPDE theory is needed to make sense of the limiting objects.} Hambly and Yang~\cite{HamblyYang2018} established existence and space--time regularity for stochastic heat equations on post-critically finite fractals; their Neumann regularity theorem supplies the continuous stochastic convolution used in this paper (see also recent works \cite{BCHOTW24,FSY26,BCCK26} for stochastic heat equations on general metric measure spaces).
Chen and Yang~\cite{CY26} developed solution theory for Wick-renormalized polynomial stochastic quantization equations on rough metric measure spaces under geometric and heat-kernel assumptions, and constructed invariant measures in their global-solution regime.
These works construct continuum fields or dynamics directly.
Our results instead identify a microscopic spin system and a joint critical scaling whose limiting observables are those of the nonlinear SPDE, with the diffusion, cubic coefficient, and noise intensity determined by the model.
\section{The Sierpiński gasket and heat-semigroup approximation}\label{sec:gasket}
We collect here the geometric and analytic facts about the Sierpiński gasket used in the sequel. Standard references are \cite{BarlowPerkins1988,Kigami2001,Strichartz2006}.

\subsection{Resistance geometry and harmonic extension}

Let $V_0=\{q_0,q_1,q_2\}$ be the vertices of the defining triangle and
$F_i(x)=(x+q_i)/2$. For a word $w$ of length $k$, write $K_w=F_wK$.
Then $m(K_w)=3^{-k}$ and $V_N=\bigcup_{|w|=N}F_wV_0$. Set
\[
 |V_N|=\frac{3^{N+1}+3}{2},\qquad
 \alpha=\frac{\log3}{\log5}<1.
\]
The canonical energy and its domain are
\[
 \mathcal E(f,f)=\lim_{N\to\infty}(5/3)^N
       \sum_{\{x,y\}\in E_N}(f(x)-f(y))^2,
 \qquad
 \mathcal F=\{f\in C(K):\mathcal E(f,f)<\infty\},
\]
where the limit is increasing and each undirected edge is counted once.
The associated resistance metric
\[
 R(x,y)=\sup_{\substack{f\in\mathcal F\\\mathcal E(f,f)>0}}
             \frac{|f(x)-f(y)|^2}{\mathcal E(f,f)}
\]
induces the usual topology of $K$.

For $g:V_N\to\mathbb R$, let $\iota_Ng$ be its energy-minimizing
harmonic extension to $K$, and identify $g$ with $\iota_Ng$ when no
confusion can arise. This map is linear, positive and constant-preserving;
on each level-$N$ cell its values are convex combinations of the three
boundary values. In particular,
\[
 \|\iota_Ng\|_\infty=\|g\|_{\infty,V_N},\qquad
 \|\iota_N(f|_{V_N})-f\|_\infty\longrightarrow0
 \quad(f\in C(K)).
\]
Harmonic orthogonality gives
$\mathcal E(\iota_Ng,v)=0$ whenever $v\in\mathcal F$ vanishes on $V_N$.
The trace identity and \eqref{eq:discrete-L} give, for
$\mathcal E_N(g,g)=-\langle g,L_Ng\rangle_{m_N}$ and its resistance metric $R_N$,
\begin{equation}\label{eq:gasket-energy-normalization}
 \mathcal E_N(g,g)=(1+3^{-N})^{-1}
                         \mathcal E(\iota_Ng,\iota_Ng),\qquad
 R_N(x,y)=(1+3^{-N})R(x,y),\quad x,y\in V_N.
\end{equation}

Assigning mass $3^{-N-1}$ to each cell--vertex incidence gives measures
converging weakly to $m$ and differing from $m_N$ in total variation by
$O(3^{-N})$. Thus $m_N\Rightarrow m$, uniformly over compact families
of test functions in $C(K)$ by finite nets.

Let $D=\operatorname{diam}_R(K)$. Self-similarity of the energy gives
$\operatorname{diam}_R(K_w)\le D(3/5)^k$ for $|w|=k$. Since there are
$3^k$ such cells, both $K$ and $(V_N,R_N)$ can be covered by at most
$C\rho^{-\alpha/(1-\alpha)}$ resistance balls of radius $\rho$, uniformly in $N$,
for $0<\rho\le1$. The resistance-volume estimates
\cite{Kigami2001,Strichartz2006} and their discrete counterparts are
\begin{equation}\label{eq:gasket-volume}
 m(B_R(x,\rho))\asymp\rho^{\alpha/(1-\alpha)},\qquad
 m_N(B_{R_N}(x,\rho))\ge c\rho^{\alpha/(1-\alpha)},
 \qquad 0<\rho\le\rho_0,
\end{equation}
uniformly in the centers and $N$. For the discrete bound, choose $k$
with $2D(3/5)^k\le\rho<2D(3/5)^{k-1}$. If $k\le N$, a level-$k$
cell containing $x$ lies in the closed ball and has graph mass
$|V_{N-k}|/|V_N|\ge c3^{-k}$; if $k>N$, the mass of $\{x\}$ suffices.
Testing the resistance formula with $\iota_N\mathbf1_{\{x\}}$ gives
\[
 R(x,y)\ge c(3/5)^N,
 \qquad x,y\in V_N,\quad x\ne y.
\]

\subsection{Heat kernels and approximation}

The form $(\mathcal E,\mathcal F)$ on $L^2(K,m)$ defines the
conservative symmetric Neumann semigroup $P_t=e^{tL}$, which is a
strongly continuous positive contraction semigroup on $C(K)$.
Its kernel $p_t(x,y)$ is jointly continuous on $(0,\infty)\times K^2$.
The classical continuum estimate \cite{BarlowPerkins1988,Kigami2001}
and its discrete counterpart read
\begin{equation}\label{eq:discrete-diagonal-bound}
 \sup_{N\ge0}\sup_{x,y\in V_N}p_t^N(x,y)
 +\sup_{x,y\in K}p_t(x,y)\le C(1+t^{-\alpha}),\qquad t>0.
\end{equation}
For the discrete bound, averaging the resistance inequality over
$B_{R_N}(x,\rho)$ gives
$|g(x)|^2\le C\rho^{-\alpha/(1-\alpha)}\|g\|_2^2+2\rho\mathcal E_N(g,g)$.
Since $\mathcal E_N(P_{t/2}^Nf,P_{t/2}^Nf)\le(et)^{-1}\|f\|_2^2$,
\[
 p_t^N(x,x)=\sup_{\|f\|_2=1}|P_{t/2}^Nf(x)|^2
 \le C\left(\rho^{-\alpha/(1-\alpha)}+\frac\rho t\right)
 \le C(1+t^{-\alpha}),
\]
where $\|\cdot\|_2=\|\cdot\|_{L^2(m_N)}$ and
$\rho=\min\{\rho_0,t^{1-\alpha}\}$. Cauchy--Schwarz and the semigroup
identity give the off-diagonal bound.

In particular, symmetry and the semigroup identity give
\[
 \int_0^T\|p_s^N(x,\cdot)\|_{L^2(m_N)}^2\,ds
 =\int_0^T p_{2s}^N(x,x)\,ds
 \le C\bigl(T+T^{1-\alpha}\bigr),
\]
uniformly in $N,x$, and the same estimate holds for $p_s$ and $m$.

We also use the generator bound
\begin{equation}\label{eq:gasket-generator-bound}
 \|L_N(\varphi|_{V_N})\|_\infty\le\|L\varphi\|_\infty,
 \qquad \varphi\in D_C(L).
\end{equation}
Indeed, for $h_{N,x}=\iota_N\mathbf1_{\{x\}}$, harmonic orthogonality
and the weak Neumann identity give
\[
 L_N(\varphi|_{V_N})(x)
 =\frac32\,3^N\int_K(L\varphi)h_{N,x}\,dm.
\]
Each incident cell contributes $3^{-N-1}$ to the integral of $h_{N,x}$,
and at most two cells meet at a vertex. Positivity of $h_{N,x}$ proves
\eqref{eq:gasket-generator-bound}.

\begin{proposition}\label{prop:gasket-approximation}
For $f\in C(K)$ and $T<\infty$,
\begin{equation}\label{eq:gasket-semigroup-convergence}
 \sup_{0\le t\le T}
 \|\iota_NP_t^N(f|_{V_N})-P_tf\|_\infty\longrightarrow0.
\end{equation}
The convergence is uniform for $f$ in compact subsets of $C(K)$.
Moreover, for $0<\delta<T<\infty$,
\begin{equation}\label{eq:gasket-kernel-convergence}
 \sup_{\delta\le t\le T}\sup_{x,y\in V_N}
 |p_t^N(x,y)-p_t(x,y)|\longrightarrow0.
\end{equation}
\end{proposition}

\begin{proof}
By \eqref{eq:gasket-energy-normalization}, the generator
$(1+3^{-N})L_N$ has the trace resistance $R|_{V_N\times V_N}$ and
speed measure $m_N$. Since $V_N\to K$ in resistance Hausdorff distance
and $m_N\Rightarrow m$, the volume estimate \eqref{eq:gasket-volume}
allows us to apply \cite[Theorem~1.3(i)]{Oishi2026} in the common
space $(K,R)$, with $s_0=s_1=\alpha/(1-\alpha)$ and $\kappa=1/2$.
This gives uniform kernel convergence on compact time intervals in
$(0,\infty)$; the time factor $1+3^{-N}\to1$ is absorbed by uniform
continuity of $p_t$. Thus \eqref{eq:gasket-kernel-convergence} holds.

Integrating the kernels against $f\,dm_N$, using $m_N\Rightarrow m$
uniformly on compact families and harmonic extension, gives
\eqref{eq:gasket-semigroup-convergence} on $[\delta,T]$.
For $\varphi\in D_C(L)$, \eqref{eq:gasket-generator-bound} gives
\[
 \sup_{0\le t\le\delta}
 \|P_t^N(\varphi|_{V_N})-\varphi|_{V_N}\|_\infty
 \le\delta\|L\varphi\|_\infty.
\]
Strong continuity of $P_t$, density of $D_C(L)$ in $C(K)$, and
contractivity extend the convergence to $[0,T]$ for every $f\in C(K)$.
Finite nets give uniformity on compact families of initial functions.
\end{proof}

\section{Fixed-heat-time hydrodynamic limit}\label{sec:hydro proof}
By the It\^o-Dynkin formula \cite[Appendix A, Lemma 1.5.1]{KipnisLandim1999}, we have for any $\varphi\in C(K)$ \[\braket{\pi^N_t,\varphi}=\braket{\pi^N_0,\varphi}+\int_0^t b^{N,\varphi}_s ds+M^{N,\varphi}_t,\]
where $M^{N,\varphi}$ is a martingale and \[b^{N,\varphi}_t:=\mathcal G_{N,\eta,\beta}\langle \pi^N,\varphi\rangle(\sigma_t^N)=\frac1{|V_N|}\sum_{x\in V_N}\varphi(x)\left[-\sigma_t^N(x)+\tanh\bigl(\beta h_{N,\eta}(\sigma_t^N,x)\bigr)\right].\]

\begin{lemma}\label{lem: hydro tight}
$\sup_{0\le t\le T}|M_t^{N,\varphi}|$ converges to zero in probability as $N\rightarrow\infty$.
Also, for any $T>0$, the laws of
$(\pi_t^N)_{0\le t\le T}$, $N\ge1$, form a tight family on
$D([0,T];\mathcal M(K))$.
\end{lemma}

\begin{proof}
    We will show that for any $\varphi\in C(K)$, the sequence of real-valued processes \((\braket{\pi^N_t,\varphi})_{N\ge 1}\) is tight in $\mathcal{D}([0,T];\mathbb{R})$. This implies the second desired result by \cite[Chapter 4, Proposition 1.7]{KipnisLandim1999} (see also \cite[Section 6.1]{CG-hydro}). 
    We claim that \begin{equation}\label{est:hydro mart BV}
        \frac{d}{dt}\braket{M^{N,\varphi}}_t\le 4\frac{\norm{\varphi}_\infty^2}{|V_N|}
    \end{equation} and \begin{equation}\label{est: hydro dift bdd}
        |b^{N,\varphi}_t|\le 2\norm{\varphi}_\infty.
    \end{equation}
    For a pure jump process, we have the carr\'e-du-champ identity \begin{align*}
    \frac{d}{dt}\braket{M^{N,\varphi}}_t=&\mathcal G_{N,\eta,\beta}(F^2)-2F\mathcal G_{N,\eta,\beta}(F)\\
    =& \sum_{x\in V_N}c_{N,\eta,\beta}(\sigma^N_t,x)[F((\sigma^N_t)^x)-F(\sigma^N_t)]^2
    \end{align*}
    with \(F(\sigma^N_t)=\braket{\pi^N_t,\varphi}\), which gives us \[F(\sigma^x)-F(\sigma)=-\frac{2}{|V_N|}\varphi(x)\sigma(x)\implies [F(\sigma^x)-F(\sigma)]^2=\frac{4\varphi(x)^2}{|V_N|^2}.\]
    Therefore, we have \[\frac{d}{dt}\braket{M^{N,\varphi}}_t=\sum_{x\in V_N}c_N(\sigma^N_t,x)\frac{4\varphi(x)^2}{|V_N|^2}\le \frac{4}{|V_N|^2}\sum_{x\in V_N}\varphi(x)^2\le \frac{4\norm{\varphi}^2_\infty}{|V_N|},\] which proves \eqref{est:hydro mart BV}.\\
    Since $0\le \norm{\tanh}_{L^\infty(\mathbb{R})} \le 1$, we have $|-\sigma^N_t(x)+\tanh(\cdots)|\le 2$, giving us directly \begin{align*}
    |b^{N,\varphi}_t|
    \le \frac{2}{|V_N|}\sum_{x\in V_N}|\varphi(x)|\le 2\norm{\varphi}_\infty,
    \end{align*} which is \eqref{est: hydro dift bdd}.\\
    By \eqref{est:hydro mart BV}, \[\langle M^{N,\varphi}\rangle_T=\int_0^T\frac{d}{dt}\langle M^{N,\varphi}\rangle_t\,dt\le\frac{4T\|\varphi\|_\infty^2}{|V_N|},\]
    which implies \[\mathbb E |M_T^{N,\varphi}|^2=\mathbb E \langle M^{N,\varphi}\rangle_T\le\frac{4T\|\varphi\|_\infty^2}{|V_N|}.\]
    Doob's martingale inequality then gives us \[\mathbb E\sup_{0\le t\le T}|M_t^{N,\varphi}|^2\le4\mathbb E |M_T^{N,\varphi}|^2\le\frac{CT\|\varphi\|_\infty^2}{|V_N|}.\]
    Thus, for every $\varepsilon>0$, we have
    \[\mathbb P\left(\sup_{0\le t\le T}|M_t^{N,\varphi}|>\varepsilon\right)\le\frac{1}{\varepsilon^2}\mathbb E\sup_{0\le t\le T}|M_t^{N,\varphi}|^2\le\frac{CT\|\varphi\|_\infty^2}{\varepsilon^2 |V_N|}\stackrel{N\uparrow\infty}{\longrightarrow}
    0.\]
    This implies the desired claim on $M$. \begin{equation}\label{lim:hydro mart lim 0}
        \sup_{0\le t\le T}|M_t^{N,\varphi}|\xrightarrow[N\uparrow\infty]{\mathbb P}0.
    \end{equation}

    We now handle the drift term, for which we introduce the notation \(A^{N,\varphi}_t:=\int_0^t b^{N,\varphi}_sds\). By \eqref{est: hydro dift bdd}, we have for \(0\le s<t\le T\)\begin{equation}\label{est:hydro drift int lip}|A_t^{N,\varphi}-A_s^{N,\varphi}|=\left|\int_s^t b_r^{N,\varphi}\,dr\right|\le\int_s^t |b_r^{N,\varphi}|\,dr\le2\|\varphi\|_\infty |t-s|.
    \end{equation} {Set \( B=\{\nu\in\mathcal M(K):\|\nu\|_{\mathrm{TV}}\le1\},\) equipped with the weak-* topology. Since $K$ is compact metrizable, $B$ is compact metrizable. Choose a countable dense family $(\varphi_j)_{j\ge1}$ in the unit ball of $C(K)$ and define
\[d_B(\mu,\nu)=\sum_{j=1}^{\infty}2^{-j}\bigl(|\langle\mu-\nu,\varphi_j\rangle|\wedge1\bigr).\]
This metric generates the weak-* topology on $B$.
By \eqref{lim:hydro mart lim 0} and \eqref{est:hydro drift int lip}, for every $\delta>0$,
\[
\begin{aligned}
&\mathbb E\sup_{\substack{s,t\le T\\|t-s|\le\delta}}
d_B(\pi_t^N,\pi_s^N)\\
&\qquad\le
2\delta+
2\sum_{j=1}^{\infty}2^{-j}
\mathbb E\sup_{t\le T}|M_t^{N,\varphi_j}|\\
&\qquad\le 2\delta+C_T|V_N|^{-1/2}.
\end{aligned}
\]
Here the last inequality follows from \eqref{est:hydro mart BV} and Doob's
inequality, uniformly in $j$.

Since $\pi_t^N\in B$ for every $t$, this proves
tightness in $D([0,T];B)$. Moreover, letting first
$N\to\infty$ and then $\delta\downarrow0$ shows that
every subsequential limit is supported on $C([0,T];B)$.
Thus the sequence is $C$-tight.}
\end{proof}

We now establish the well-posedness of \eqref{eq:nonlocal-hydrodynamic} by a routine fixed point argument.

\begin{proposition}\label{prop: hydro lim WP}
    A function $u\in C([0,T],C(K))$ is a mild solution to \eqref{eq:nonlocal-hydrodynamic} with initial condition $u_0\in C(K)$ if it satisfies the point-wise mild formulation \begin{equation}\label{eq:hydro nonloc mild}
        u_t(r)=u_0(r)+\int_0^t -u_s(r)+\tanh(\beta P_\eta u_s(r))ds
    \end{equation} for every $t\in [0,T],r\in K$. For every $u_0\in C(K)$ and any $T>0$, there exists a unique mild solution $u$ to \eqref{eq:nonlocal-hydrodynamic}.
\end{proposition}
\begin{proof}
    For shorthand, we define \(F(u_t(r)):=-u_t(r)+\tanh(\beta P_\eta u_t(r))\), so \eqref{eq:hydro nonloc mild} can be rewritten as \begin{equation}\label{eq:hydro mild pt F}
        u_t(r)=u_0(r)+\int_0^t F(u_s(r))ds.
    \end{equation}
    We will first show that a fixed point for the Picard iteration \[u^0_t=u_0,\quad u^{n}_t={u^{0}_t}+\int_0^tF(u^{n-1}_s)ds\]exists in $C([0,T],C(K))$ for sufficiently small $T>0$. 
    Since $|\tanh(\beta x)|\le \beta|x|$, we have {\[\norm{F(u_1)-F(u_2)}_\infty\le \norm{u_1-u_2}_\infty+\beta \norm{P_\eta u_1-P_\eta u_2}\le (1+\beta)\norm{u_1-u_2}_\infty,\]}where in the last step we used $P_\eta$ is an $L^\infty(K)$ contraction. 
    Taking another $L^\infty$ norm in $t$, we see that the map $u(\circ,\cdot)\to \int_0^\circ F(u(s,\cdot))ds$ is a contraction
    on $C([0,T];C(K))$ for $T<(1+\beta)^{-1}$, giving us the desired fixed point. 
    To obtain a solution for $t\in [0,T]$ for arbitrarily large $T$, simply partition $[0,T]$ into intervals of maximal size $<(1+\beta)^{-1}$ and re-run the construction over each sub-interval.\\
    We now show uniqueness. Let $u,v$ be two weak solutions of \eqref{eq:hydro nonloc mild} with initial condition $u_0\in C(K)$ in $C([0,T],C(K))$and set $\delta_t(r):=u_t(r)-v_t(r)$. By direct computation, we have
    \[\delta_t=\int_0^t-\delta_s +\tanh(\beta P_\eta u_s)-\tanh(\beta P_\eta v_s) ds,\]
    which implies \[\norm{\delta_t}_\infty\le \int_0^t \norm{\delta_s}_\infty +\norm{\tanh(\beta P_\eta u_s)-\tanh(\beta P_\eta v_s)}_\infty ds.\]
    Since $x\mapsto\tanh(\beta x)$ is $\beta-$Lipschitz and $P_\eta$ is an $L^\infty(K)$ contraction, we have for all $r\in K$ \begin{align*}
        &|\tanh(\beta P_\eta u_s)-\tanh(\beta P_\eta v_s)|\\
        \le& \beta|P_\eta u_s-P_\eta v_s|\leq \beta P_\eta \norm{\delta_s}_\infty\le \beta \norm{\delta_s}_\infty.
    \end{align*}
    This gives us \[\norm{\delta_t}_\infty\le \int_0^t (1+\beta)\norm{\delta_s}_\infty ds,\] which implies by Gronwall's inequality \(\norm{\delta_t}_\infty=0\).
    Therefore $u= v$ in $C([0,T],C(K))$, as desired.
\end{proof}

\begin{proposition}\label{prop:hydro drift conv}
    Fix \(T>0\) and suppose the initial conditions satisfy the assumptions of Theorem \ref{thm:hydro-lim}. Let \(B=\{\nu\in\mathcal M(K):\|\nu\|_{\mathrm{TV}}\le1\},\) equipped with the weak-\(*\) topology. Every subsequential limit in law of \((\pi_t^N)_{0\le t\le T}\) in \(D([0,T];B)\) is supported on \(C([0,T];B)\) and coincides almost surely with the deterministic path
    $$\pi_t=u_tm,\qquad 0\le t\le T,$$where \(u\) is the unique solution of the nonlocal equation \eqref{eq:nonlocal-hydrodynamic} with initial condition \(u_0\).
\end{proposition}

\begin{proof}
    Any limiting signed Radon measure $\pi_t$ being absolutely continuous with respect to $m$ follows from \[-m_N\le \pi^N_t\le m_N\] for all $N\ge 0, t\in \mathbb{R}_+$, where \(m_N:=|V_N|^{-1}\sum_{x\in V_N}\delta_{x}\) and the $\le$ is standard notation for comparison of signed measures, since passing to the limit yields \[-m\le \pi_t\le m.\]
    {By Lemma 1, the laws of the empirical-measure paths are $C$-tight in $D([0,T];B)$. Passing to a subsequence, let $\pi^N\Rightarrow\pi$, where $\pi\in C([0,T];B)$ almost surely.

By the Skorokhod representation theorem, we may work on a probability space on which this convergence holds almost surely. Since the limiting paths are continuous,
convergence in the $J_1$ topology implies
\[
\sup_{t\le T}d_B(\pi_t^N,\pi_t)\longrightarrow0
\qquad\text{almost surely}.
\]
Consequently, for every $\psi\in C(K)$,
\[
\sup_{t\le T}
|\langle\pi_t^N-\pi_t,\psi\rangle|
\longrightarrow0
\qquad\text{almost surely}.
\]}
    For any $t>0$, set $u_t:=\frac{d\pi_t}{dm}$, which by Radon-Nikodym we know must be $\le 1$.
    Since for all $\varphi\in C(K)$ we have $M^{N,\varphi}\stackrel{N\uparrow+\infty}{\longrightarrow}0$ over all subsequences by Lemma \ref{lem: hydro tight}, it remains to show that in any subsequence with limiting $m-$density $u_t$, we have as real-valued processes in $t$ \begin{equation}\label{lim:hydro drift conv}
        b^{N,\varphi}_t\stackrel{N\uparrow\infty}{\longrightarrow} \int_K \varphi(r)[-u_t(r)+\tanh(\beta P_\eta u_t(r))] m(dr),
    \end{equation} as this means the limiting density satisfies, for every $\varphi\in C(K)$,
    \[\int_K \varphi(r)u_t(r) m(dr)=\int_K u_0(r)\varphi(r)m(dr)+\int_0^t\int_K \varphi(r)[-u_s(r)+\tanh(\beta P_\eta u_s(r))] m(dr)ds,\] which is the weak formulation of \eqref{thm:hydro-lim}.\\
    Convergence of the initial conditions is assumed. Since $u_t$ is presumed to be the density of the subsequential limit, the linear term convergence is free, leaving us to show \[\left||V_N|^{-1}\sum_{x\in V_N}\varphi(x)\tanh(\beta h_{N,\eta}(\sigma^N_t,x))-\int_{K }\varphi(r) \tanh(\beta P_\eta u_t(r))m(dr)\right|\stackrel{N\uparrow\infty}{\longrightarrow}0.\]
    This can be further simplified elementarily: \begin{align*}
    &\left|\frac1{|V_N|}\sum_{x\in V_N}\varphi(x)\tanh\bigl(\beta h_{N,\eta}(\sigma_t^N,x)\bigr)-\int_{K}\varphi(r)\tanh(\beta P_\eta u_t(r))] m(dr)\right|\\
    \le &\left|\frac1{|V_N|}\sum_{x\in V_N}\varphi(x)\tanh\bigl(\beta h_{N,\eta}(\sigma_t^N,x)\bigr)-\frac{1}{|V_N|}\sum_{x\in V_N}\varphi(x)\tanh(\beta P_\eta u_t(x))]\right|\\
    &+\left|\frac1{|V_N|}\sum_{x\in V_N}\varphi(x)\tanh\bigl(\beta P_\eta u_t(x)\bigr)-\int_{K}\varphi(r)\tanh(\beta P_\eta u_t(r))] m(dr)\right|=:D_1^N+D_2^N.
\end{align*}
    {The convergence of $D_2^N$ to 0 is uniform in time. indeed, since $\pi\in C([0,T];B)$ and the heat kernel $p_\eta$ is jointly continuous, the map $t\mapsto P_\eta\pi_t$ is continuous into $C(K)$.
    Consequently,\[\left\{\varphi(\cdot)\tanh\bigl(\beta P_\eta\pi_t(\cdot)\bigr):0\le t\le T\right\}\]
is compact in $C(K)$. Weak convergence $m_N\Rightarrow m$ is uniform on compact families of continuous functions, and hence\[\sup_{t\le T}|D_2^N(t)|\longrightarrow0.\] }
    The problem is thus further reduced to showing \[D_1^N:=\left|\frac1{|V_N|}\sum_{x\in V_N}\varphi(x)\left[\tanh\bigl(\beta h_{N,\eta}(\sigma_t^N,x)\bigr)-\tanh(\beta P_\eta u_t(x))\right]\right|\stackrel{N\uparrow\infty}{\longrightarrow} 0.\]
    Since $z\mapsto \tanh(\beta z)$ is $\beta-$Lipschitz, we have \[\left|\tanh\bigl(\beta h_{N,\eta}(\sigma_t^N,x)\bigr)-\tanh(\beta P_\eta u_t(x))\right|\le \beta |h_{N,\eta}(\sigma_t^N,x)- P_\eta u_t(x)|.\]
    Thus, it remains to prove \begin{equation}\label{lim:hydro heat ker replacement}
        \sup_{0\le t\le T,x\in V_N}|h_{N,\eta}(\sigma_t^N,x)- P_\eta u_t(x)|\to 0.
    \end{equation}
    Recall by\eqref{eq:heat-kernel-Kac-field} that \(
    h_{N,\eta}(\sigma_t^N,x)=\left\langle\pi_t^N,p_\eta^N(x,\cdot)\right\rangle-\frac1{|V_N|}p_\eta^N(x,x)\sigma_t^N(x),\) so
    \begin{align*}
        \left|h_{N,\eta}(\sigma_t^N,x)-P_\eta u_t(x)\right|=&\left|\braket{\pi^N_t,p^N_\eta(x,\cdot)}-\braket{u_tm,p_\eta(x,\cdot)}-\frac{1}{|V_N|}p^N_\eta(x,x)\sigma^N_t(x)\right|\\
    &\le \left|\braket{\pi^N_t,p^N_\eta(x,\cdot)-p_\eta(x,\cdot)}\right|\\
    &+\left|\braket{\pi^N_t-u_tm,p_\eta(x,\cdot)}\right|+\frac1{|V_N|}p_\eta^N(x,x).
    \end{align*}
By \eqref{eq:discrete-diagonal-bound}, for fixed $\eta>0$,
\begin{equation}\label{conv:single site}
    \sup_{x\in V_N}\frac{p_\eta^N(x,x)}{|V_N|}
    \le\frac{C(1+\eta^{-\alpha})}{|V_N|}
    \stackrel{N\uparrow\infty}{\longrightarrow}0.
\end{equation}
    Denote the total variation norm on signed measures on $K$ by $\norm{\cdot}_{\mathrm{TV}}$. 
    Since\(\|\pi_t^N\|_{\mathrm{TV}}=1,\) we have
    \[\left|\braket{\pi^N_t,p^N_\eta(x,\cdot)-p_\eta(x,\cdot)}\right|\le\sup_{y\in V_N}\left|p_\eta^N(x,y)-p_\eta(x,y)\right|.\]
    The RHS quantity goes to zero by Proposition \ref{prop:gasket-approximation}, giving us uniformly in $t$    \begin{equation}\label{conv: heat ker diff pi}
    \sup_{x\in V_N}\braket{\pi_t^N,p_\eta^N(x,\cdot)-p_\eta(x,\cdot)}\stackrel{N\uparrow\infty}{\longrightarrow}0.\end{equation}

    For the final term, since \(\eta>0\) is fixed, the continuum heat kernel \(p_\eta(x,y)\) is jointly continuous on the compact set \(K\times K\). 
    Consequently, \(\mathcal K_\eta:=\{p_\eta(x,\cdot):x\in K\}\) is compact in \(C(K)\). 
    Let \(\varepsilon>0\). By compactness of $\mathcal K_\eta$, there exist \(x_1,\dots,x_M\in K\) such that for every \(x\in K\), some \(x_j\) satisfies\[\|p_\eta(x,\cdot)-p_\eta(x_j,\cdot)\|_\infty<\varepsilon.\]
    Since \(\|\pi_t^N-u_tm\|_{\mathrm{TV}}\le 2,\)
    we obtain{\[\sup_{0\le t\le T\\x\in K}|\braket{\pi^N_t-u_tm,p_\eta(x,\cdot)}|\le 2\varepsilon+\max_{1\le j\le M}\sup_{0\le t\le T}|\braket{\pi^N_t-u_tm,p_\eta(x_j,\cdot)}|.\]}
    The maximum on RHS tends to zero by the assumed weak convergence of $\pi^N_t$ to $\pi_t$ along this subsequence, since it contains only finitely many fixed continuous test functions. 
    Thus \[\limsup_{N\to\infty}\sup_{x\in K}\left|\left\langle\pi_t^N-u_tm,p_\eta(x,\cdot)\right\rangle\right|\le 2\varepsilon .\]
    Since \(\varepsilon>0\) is arbitrary, we have uniformly in $t$ \begin{equation}\label{conv:meas term}
    \sup_{x\in K} |\braket{\pi_t^N-u_tm,p_\eta(x,\cdot)}|\stackrel{N\uparrow\infty}{\longrightarrow}0.
    \end{equation}
    Since the results of \eqref{conv:single site}, \eqref{conv: heat ker diff pi}, and \eqref{conv:meas term} do not depend on $t$, we can combine them to obtain \eqref{lim:hydro heat ker replacement}, which concludes the convergence proof.\\
    {It remains to show that the limiting density satisfies $u\in C([0,T],C(K))$ for any $T>0$. For every $\varphi\in C(K)$, variation of constants of standard ODE theory gives us \[\braket{\pi_t, \varphi}=e^{-t}\braket{u_0m,\varphi}+\int_0^t e^{-(t-s)}\braket{\tanh(\beta P_\eta \pi_s)m,\varphi}ds,\]
    thus the density $u$ satisfies for all $t\in [0,T]$ \[u_t(r)=e^{-t}u_0(r)+\int_0^t e^{-(t-s)}\tanh(\beta P_\eta \pi_s(r)) ds.\]
    The desired conclusion follows from weak continuity of $\pi_s$, compactness of the family $\set{p_\eta(x,\cdot):x\in K}$ in $C(K)$ and the joint continuity of $p_\eta(\cdot,\cdot)$, since that shows the map $(s,r)\mapsto P_\eta \pi_s(r)$ is continuous.}
\end{proof}

Combining Lemma \ref{lem: hydro tight}, Proposition \ref{prop: hydro lim WP} and Proposition \ref{prop:hydro drift conv} immediately yields Theorem \ref{thm:hydro-lim}.

\section{The Allen-Cahn limit of the rescaled hydrodynamic equation}\label{sec:det AC lim proof}

We first show that the Allen-Cahn equation is also well posed.

\begin{proposition}\label{prop:Allen-Chan WP}
Let \(P_t=e^{tL}\) be a strongly continuous Markov semigroup on
\(C(K)\). For every \(X_0\in C(K)\), \eqref{eq:allen-cahn}
admits a unique global mild solution
\[
    X\in C([0,\infty);C(K)).
\]
Moreover, we have for all $t\ge 0$\begin{equation}\label{est:allen-cahn unif bdd}
    \norm{X(t,\cdot)}_\infty\leq \max\{\norm{X_0}_\infty,\sqrt{3A_+}\},
\end{equation} where \(A_+=\max\{A,0\}\).
\end{proposition}

\begin{proof}
Since \(F(f):=Af-\frac13f^3\) is locally Lipschitz on \(C(K)\), \eqref{eq:allen-cahn}, which admits the mild form
\[
    X(t,\cdot)=P_tX_0+\int_0^tP_{t-s}F(X(s,\cdot))ds
\]
has a unique maximal mild solution \(X\in C([0,T_*);C(K))\), where \(T_*\in(0,\infty]\) is the maximal existence time. Moreover, the blow-up alternative gives
\[
    T_*<\infty \quad\Longrightarrow\quad
    \lim_{t\uparrow T_*}\|X(t,\cdot)\|_\infty=\infty.
\]
We will show $T_*=\infty$ by showing \eqref{est:allen-cahn unif bdd} in two steps.



\noindent
\emph{Step 1: comparison with spatially constant solutions.} Recall that $X_0\in C(K)$ and $K$ is compact, so we can set
\[
    m_-:=\min_{r\in K}X_0(r),
    \qquad
    m_+:=\max_{r\in K}X_0(r).
\]
Let \(y_-\) and \(y_+\) solve the scalar ODE \begin{equation}\label{eq:cubic ODE}
    y'(t)=Ay(t)-\frac13y(t)^3
\end{equation}
with initial conditions
\[
    y_-(0)=m_-,
    \qquad
    y_+(0)=m_+.
\]
Since \(P_t1=1\), the spatially constant functions
\[
    r\longmapsto y_-(t),
    \qquad
    r\longmapsto y_+(t)
\]
are mild solutions of \eqref{eq:allen-cahn}.

We claim that $y_-(t)\le X(t,r)\le y_+(t)$
for every $r\in K$ and $t<T_*$.

If \(u(0,\cdot)\leq v(0,\cdot)\) and
\(u,v\) are two mild solutions, then \(w=v-u\) satisfies
\[
    w(t,\cdot)
    =
    P_tw(0,\cdot)
    +
    \int_0^t
    P_{t-s}\bigl(a(s,\cdot)w(s,\cdot)\bigr)\,ds,
\]
where
\[
    a(s,r)
    =
    A-\frac13
    \left(
        v_s(r)^2
        +
        u_s(r)v_s(r)
        +
        u_s(r)^2
    \right).
\]
On every compact time interval on which \(u\) and \(v\) exist,
\(a\) is bounded. Choosing \(C>0\) so that \(a+C\geq0\), and setting
\[
    \widetilde w(t,\cdot):=e^{Ct}w(t,\cdot),
\]
one obtains
\[
    \widetilde w(t,\cdot)
    =
    P_tw(0,\cdot)
    +
    \int_0^t
    P_{t-s}
    \bigl((a(s,\cdot)+C)\widetilde w(s,\cdot)\bigr)\,ds.
\]
Since \(P_t\) is positivity preserving, Picard iteration yields
\(\widetilde w(t,\cdot)\ge0\), and hence \(u_t\leq v(t,\cdot),\) as required.

\medskip

\noindent
\emph{Step 2: the ODE \eqref{eq:cubic ODE} is bounded for all time.}

Let \(y:\mathbb{R}_+\to\mathbb{R}\) solve \eqref{eq:cubic ODE}.
If \(y>\sqrt{3A_+},\) then\[y'=y\left(A-\frac13y^2\right)
    \leq0\]
Similarly, if \( y<-\sqrt{3A_+},\)
then
\[
    y'\geq0.
\]
Therefore
\[
    |y(t)|
    \leq
    \max\left\{
        |y(0)|,\sqrt{3A_+}
    \right\}
\]
for all \(t\geq0\), as required.

The conclusion of step 2 applied to \(y_-\) and \(y_+\) and further combined with the comparison from step 1 gives us
\[\|X(t,\cdot)\|_\infty\leq\max\left\{\|X_0\|_\infty,\sqrt{3A_+}\right\}\]
for every \(t<T_*\). Thus the maximal solution remains uniformly bounded on
\([0,T_*)\) independent of $T_*$, so \(T_*=\infty\), as required.
\end{proof}

\begin{proof}[Proof of Theorem \ref{thm:allen-cahn-limit}]
By direct scaling computation, we have $X^\eta$ satisfying \[\partial_tX^\eta=\eta^{-\frac{3}{2}}\left[-\eta^{\frac{1}{2}}X^\eta+\tanh(\beta_\eta \eta^{\frac{1}{2}}P_\eta X^\eta)\right].\]
Adding and subtracting \(P_\eta X/\eta\) gives
\[
\partial_tX^\eta= \frac{P_\eta-I}{\eta}X^\eta+F_\eta (X^\eta),\qquad
F_\eta(Y):=\frac{\tanh(\beta_\eta\sqrt\eta\,P_\eta Y)-\sqrt\eta\,P_\eta Y}{\eta^{3/2}}.
\]
We thus obtain a mild formulation of $X^\eta$: \[X^\eta(t,r)=S^\eta_tX_0^\eta(r)+\int_0^t S^\eta_{t-s}F_\eta(X^\eta(s,r)) ds,\]
where \(S^\eta_t:=e^{tL_\eta},\quad L_\eta:=\frac{P_\eta-I}{\eta}.\)\\

For the desired result, we need the following statements. \begin{enumerate}
    \item Uniformly on compact subsets of $C(K)$, \[F_\eta (Y)\stackrel{\eta\downarrow0}{\longrightarrow}AY-\frac{Y^3}{3}.\]
    For every $R>0$, there exists $C_R$ independent of sufficiently small $\eta$ such that if $\norm{Y}_\infty,\norm{Z}_\infty\le R$, \[\norm{F_\eta(Y)-F_\eta(Z)}\le C_R\norm{Y-Z}_\infty.\]
    \item As semigroups on $C(K)$, \(S^\eta\stackrel{\eta\downarrow0}{\longrightarrow} P\) locally uniformly in time.
    \item The family of spacetime functions $(X^\eta)_{\eta>0}$ is uniformly bounded in $C([0,T];C(K))$ for any $T>0$.
\end{enumerate}
By the Taylor expansion \(\tanh(z)=z-\frac{z^3}{3}+O(z^5)\) and $\beta_\eta=1+A\eta+o(\eta)$, we have 
\[F_\eta(Y)=\frac{\beta_\eta-1}{\eta}P_\eta Y-\frac{\beta_\eta^3}{3}(P_\eta Y)^3+R_\eta Y,\text{ where } \norm{R_\eta Y}_\infty \lesssim \eta \norm{Y}_\infty^5\] for $Y$ belonging to a bounded subset of $C(K)$, where the implicit constant is uniform over any such bounded subset. A standard $\varepsilon$--net over any compact subset $\mathcal{K}\subset C(K)$ argument then gives, for all such $\mathcal{K}$, \(\sup_{Y\in \mathcal{K}}\norm{P_\eta Y-Y}_\infty\stackrel{\eta\downarrow0}{\longrightarrow} 0\). The first claim of item~1 then follows from \(\frac{\beta_\eta-1}{\eta}\stackrel{\eta\downarrow0}{\longrightarrow}A,\quad \beta_\eta^3\stackrel{\eta\downarrow0}{\longrightarrow}1\) and the boundedness of $\mathcal{K}$. 
For the second claim, by definition we have $F_\eta(Y)=\phi_\eta(P_\eta Y)$, where $\phi_\eta(x):=\frac{\tanh(\beta_\eta \sqrt{\eta}x)-\sqrt{\eta}x}{\eta^{3/2}}$. A direct computation gives \[\phi_\eta'(x)=\frac{\beta_\eta\sech^2(\beta_\eta \sqrt{\eta}x)-1}{\eta}=\frac{(\beta_\eta-1)-\beta_\eta\tanh^2(\beta_\eta \sqrt{\eta}x)}{\eta}.\] Since $x\mapsto\tanh(\beta x)$ is $\beta$--Lipschitz, we have for $|x|\le R$ \[|\phi_\eta'(x)|\le \frac{|\beta_\eta-1|}{\eta}+\beta_\eta^3x^2\le C_R\] where $C_R$ is uniform in small $\eta$ since $(\beta_\eta-1)/\eta\stackrel{\eta\downarrow0}{\longrightarrow}A$. This along with $P_\eta$ being an $L^\infty$ contraction yields the second claim of item 1.\\
For item~2, since
$\|(P_\eta f-f)/\eta-Lf\|_\infty\to0$ as $\eta\downarrow0$
for every $f\in D_C(L)$, the desired result follows from
the Trotter--Kato approximation theorem
\cite[Chapter III, Theorem 4.8]{EngelNagel2000}.\\
For item 3, by the fact that $S^\eta$ is strongly contractive in $C(K)$ and $|\tanh(z)|\le |z|$ we have {\[\norm{X^\eta_t}_\infty\le e^{-\frac{t}{\eta}}\norm{X_0}_\infty+\frac{\beta_\eta}{\eta}\int_0^t e^{-\frac{(t-s)}{\eta}}\norm{X^\eta_s}_\infty ds.\] Gronwall after multiplication on both sides by $e^{t/\eta}$ then gives us \(\norm{X^\eta_t}_\infty\le e^{t\left(\frac{\beta_\eta-1}{\eta}\right)}\norm{X_0}_\infty\le C_T\norm{X_0}_\infty\) for $0\le t\le T$ since $(\beta_\eta-1)/\eta\stackrel{\eta\downarrow0}{\longrightarrow}A$, as desired.}\\
By item 1, we have in the bounded $C(K)$ set of item 3 $F_\eta$ is uniformly Lipschitz. 
{We will show that the set $\set{X_t:0\le t\le T}$ is compact in $C(K)$, from which item 1 will imply that $F_\eta$ converges uniformly to $F$ on this set}. 
Recall that \eqref{eq:allen-cahn} has the mild formulation \[X(t,r)=P_tX_0(r)+\int_0^tP_{t-s}F(X(s,r))ds.\]
This immediately gives us \[|X^\eta(t,r)-X(t,r)|\le |S^\eta_tX^\eta_0(r)-P_tX_0(r)| +\left|\int_0^t S^{\eta}_{t-s}F_\eta(X^\eta(s,r))-P_{t-s}F(X(s,r)) ds\right|.\]

For the initial condition term, we have \[
\begin{aligned}
\|S_t^\eta X^\eta_0-P_tX_0\|_\infty
&\leq\|S_t^\eta(X^\eta_0-X_0)\|_\infty+\|(S_t^\eta-P_t)X_0\|_\infty  \\
&\leq\|X^\eta_0-X_0\|_\infty+\|(S_t^\eta-P_t)X_0\|_\infty .
\end{aligned}
\]
Thus by items 2 and 3,\[\sup_{t\in[0,T]}\|S_t^\eta X^\eta_0-P_tX_0\|_\infty\stackrel{\eta\downarrow0}{\longrightarrow} 0.\]

For the integral term, we will need to show that $S^\eta$ is a contraction in $C(K)$ for all $\eta$. Recall that $P$ is a $C(K)$ contraction, and we have by direct computation \begin{align*}
    S^\eta_t:=&\exp\left(\frac{t(P_\eta-I)}{\eta}\right)=e^{-\frac{t}{\eta}}\sum_{n\ge 0}\frac{(t/\eta)^n}{n!}P_{n\eta}\implies\\
    \norm{S^\eta_t f}_\infty\le&e^{-\frac{t}{\eta}}\sum_{n\ge 0}\frac{(t/\eta)^n}{n!}\norm{P_{n\eta}f}_\infty\le \norm{f}_\infty e^{-\frac{t}{\eta}}\sum_{n\ge 0}\frac{(t/\eta)^n}{n!}=\norm{f}_\infty.
\end{align*}

We now add and subtract to the integrand
\(S_{t-s}^\eta F_\eta(X(s,\cdot))\) and
\(S_{t-s}^\eta F(X(s,\cdot))\) to obtain
\[
\begin{aligned}
&\|S_{t-s}^\eta F_\eta(X^\eta(s,\cdot))
      -P_{t-s}F(X(s,\cdot))\|_\infty \\
&\qquad\leq
\|S_{t-s}^\eta[F_\eta(X^\eta(s,\cdot))-F_\eta(X(s,\cdot))]\|_\infty \\
&\qquad\quad+\|S_{t-s}^\eta[F_\eta(X(s,\cdot))-F(X(s,\cdot))]\|_\infty \\
&\qquad\quad+\|(S_{t-s}^\eta-P_{t-s})F(X(s,\cdot))\|_\infty .
\end{aligned}
\]
Since \(S_t^\eta\) is a $C(K)$ contraction and, by item 3, \(X^\eta(s,\cdot)\) and \(X(s,\cdot)\) remain in a common bounded subset of \(C(K)\), item 1 implies that for some \(C_T<\infty\),
\[\|F_\eta(X^\eta(s,\cdot))-F_\eta(X(s,\cdot))\|_\infty\leq
C_T\|X^\eta(s,\cdot)-X(s,\cdot)\|_\infty .\]
Moreover,
\[\sup_{s\in[0,T]}\|F_\eta(X(s,\cdot))-F(X(s,\cdot))\|_\infty\stackrel{\eta\downarrow0}{\longrightarrow}0.\]
Finally, since \(s\mapsto F(X(s,\cdot))\) is continuous,
\(\mathcal K_T:=\{F(X(s,\cdot)):s\in[0,T]\}\) is compact in \(C(K)\). 
Item 2, together with the uniform boundedness of \(S_t^\eta\) and \(P_t\) in $t$, therefore implies
\[
\sup_{\substack{0\leq s\leq t\leq T}}\|(S_{t-s}^\eta-P_{t-s})F(X(s,\cdot))\|_\infty\stackrel{\eta\downarrow0}{\longrightarrow}0.
\]
Consequently, there exists \(\varepsilon_\eta(T)\to0\) such that
\[\|X^\eta(t,\cdot)-X(t,\cdot)\|_\infty\leq\varepsilon_\eta(T)+C_T\int_0^t\|X^\eta(s,\cdot)-X(s,\cdot)\|_\infty\,ds
\]
for every \(t\in[0,T]\). Gronwall's inequality gives
\[\sup_{t\in[0,T]}\|X^\eta(t,\cdot)-X(t,\cdot)\|_\infty\leq\varepsilon_\eta(T)e^{C_TT}\stackrel{\eta\downarrow0}{\longrightarrow}0.
\]
Hence \(X^\eta\stackrel{\eta\downarrow0}{\longrightarrow} X\quad\text{in }C([0,T];C(K)),
\) as required.
\end{proof}

\section{The stochastic Allen--Cahn limit}\label{sec:phi4 lim proof}

We prove Theorem~\ref{thm:stochastic-limit} by comparing the mild
equations for the interaction-averaged empirical field and the limiting
field $\Phi$.
The comparison is first made for truncated equations, whose drifts
and noise covariances can be controlled without an a priori bound on
the microscopic field. Convergence of these equations and global continuity of $\Phi$
will then allow us to remove the truncation.

We must first define what is meant by a solution of \eqref{eq:phi4}. Let $\xi$ be the space--time white noise on $K$, that is, it is the unique isonormal Gaussian process on $L^2(\mathbb{R}_+\times K)$ (see ex. \cite[Section 5]{CY26}). We write
\[
 W_t(f)=\int_0^t\int_K f(r)\,\xi(dr,ds),\qquad f\in L^2(K,m).
\]
Thus $W$ is a cylindrical Wiener process, with
$\mathbb E[W_t(f)W_u(g)]=(t\wedge u)\langle f,g\rangle_m$.
Set $F(u)=Au-\frac13u^3$ and define the stochastic convolution
\begin{equation}\label{eq:limiting-stochastic-convolution}
 Z_t(r)=\sqrt{\frac2\theta}\int_0^t\int_K
       p_{t-s}(r,r')\,\xi(dr',ds).
\end{equation}
The heat-kernel estimate \eqref{eq:discrete-diagonal-bound} makes
this stochastic integral well-defined in the sense of
Walsh~\cite{Walsh86}. By the Neumann case of
Hambly--Yang \cite[Theorem~1.2]{HamblyYang2018}, $Z$ has a continuous
adapted version on $[0,\infty)\times K$, with $Z_0=0$.
We use this version throughout.

\begin{proposition}\label{thm:phi4 wp}
For every $\Phi_0\in C(K)$, equation \eqref{eq:phi4} has a unique global adapted mild solution with paths in $C([0,\infty);C(K))$, defined as satsifying the integral equation
\begin{equation}\label{eq:phi4-mild}
 \Phi_t(r)=(P_t\Phi_0)(r)
 +\int_0^t\bigl[P_{t-s}F(\Phi_s)\bigr](r)\,ds+Z_t(r).
\end{equation}
\end{proposition}

\begin{proof}
    Write $\Phi=Y+v$, where $Y_t(r):=P_t\Phi_0(r)+Z_t(r)$.
    By \cite[Theorem 1.2]{HamblyYang2018} $Y\in C([0,T];C(K))$ for all $T>0$ almost surely, so the random variable \[R_T:=\sup_{t\in [0,T]}\norm{Y(t,\cdot)}_\infty\] is finite almost surely.
    Since $P_t$ is an $L^\infty$ contraction and $F$ is locally Lipschitz, we have, almost surely for some $T^*<\infty$ and all $0<T<T^*$, \eqref{eq:phi4-mild} has a unique solution on $C([0,T],C(K))$ by a similar Picard iteration argument as in the proof of Proposition \ref{prop: hydro lim WP} since the right hand side of the mild formulation satisfies \[\norm{Y_t+\int_0^t P_{t-s}F(\Phi_s)ds}_\infty\le R_T+\int_0^t|A|\norm{\Phi_s}_\infty+\frac{1}{3}\norm{\Phi}_\infty^3 ds\] for all $0<t<T<T^*$, and any two $Y,Z$ in a ball of radius $B$ in $C(K)$ satisfy \[\norm{F(Y)-F(Z)}_\infty \le (|A|+B^2)\norm{Y-Z}_\infty. \]
    {Abusing notation, define $T^*\in (0,\infty]$ to be the maximal lifetime of the solution $\Phi$, so that \[T^*<\infty\implies \limsup_{t\uparrow T^*}\norm{\Phi_t}_\infty=\infty.\] We will show} show $T^*=\infty$. For any $T>0$, set \(A_+:=\max\{A,0\},\quad M_T:=R_T+\sqrt{3A_+}\) and define for $0\le t \le T$ \[\overline{\Phi}_t:=Y_t+M_T,\qquad \underline{\Phi}_t:=Y_t-M_T.\]
    By definition, we have for all $r\in K$ \[\overline{\Phi}_t(r)\ge \sqrt{3A_+},\qquad \underline{\Phi}_t(r)\le -\sqrt{3A_+}.\]
    Since \(
    F(u)=Au-\frac13u^3
        =u\left(A-\frac13u^2\right)\), it follows that\[F(\overline\Phi_t(r))\leq0,\qquad F(\underline\Phi_t(r))\geq0\]
Since \(P_t1=1\), we have
\[
    \overline\Phi_t
    =P_t(\Phi_0+M_T)+Z_t,
    \qquad
    \underline\Phi_t
    =P_t(\Phi_0-M_T)+Z_t.
\]
Consequently,
\[
    \overline\Phi_t
    \geq
    P_t(\Phi_0+M_T)+Z_t
    +\int_0^tP_{t-s}F(\overline\Phi_s)\,ds,
\]
while
\[
    \underline\Phi_t
    \leq
    P_t(\Phi_0-M_T)+Z_t
    +\int_0^tP_{t-s}F(\underline\Phi_s)\,ds.
\]
Thus \(\overline\Phi\) and \(\underline\Phi\) are respectively a mild super- and subsolution of \eqref{eq:phi4-mild}. 
Since \[\Phi_0-M_T\leq\Phi_0\leq\Phi_0+M_T,\]the comparison principle gives, on \([0,T\wedge T^*)\),\[Y_t-M_T\leq\Phi_t\leq Y_t+M_T.\]
In particular,\[\sup_{0\leq t<T\wedge T^*}\|\Phi_t\|_\infty\leq R_T+M_T=2R_T+\sqrt{3A_+}<\infty\]
Since the right hand side is bounded for any $T>0$ almost surely, we have \(T^*=\infty\) almost surely.
\end{proof}

\begin{remark}
    The above proof is in some sense an easier version of the arguments in \cite{DPD,MW17,TW18,CY26}. Crucially, the result here generalizes to {compact and strongly} recurrent Dirichlet spaces that admit a {continuous} heat kernel {such that the SPDE framework of Walsh \cite{Walsh86} can be formulated}, where the proof requires no changes.
\end{remark}

\subsection{The local field and truncation}

For the remainder of this section, write $\eta=\eta_N$,
$\beta=\beta_{\eta_N}$, $\sigma^N=\sigma^{N,\eta_N}$ and $n=|V_N|$,
and assume the scaling in Theorem~\ref{thm:stochastic-limit}.
The rescaled spins have supremum norm $\eta^{-1/2}$, so we instead
study the local field entering the flip rates. Define
\[
 H_t^N(x)=\eta^{-1/2}h_{N,\eta}(\sigma^N_{t/\eta},x),
 \qquad x\in V_N,
\]
and use the same notation for its harmonic extension to $K$.
Define the diagonal operator $D_N$ and the interaction operator $K_N$ by
\[
 D_Ng(x)=\frac{p_\eta^N(x,x)}n g(x),\qquad
 K_N=P_\eta^N-D_N.
\]
Then $H_t^N=\eta^{-1/2}K_N\sigma^N_{t/\eta}$ on $V_N$.
The matrix $K_N$ is symmetric, has nonnegative entries and has row
sums at most one.

Let $M^N$ be the martingale part of $\eta^{-1/2}\sigma^N_{t/\eta}$,
starting at zero. A flip at $x$ changes this vector by
$-2\eta^{-1/2}\sigma^N_{(t/\eta)-}(x)\mathbf1_{\{x\}}$ and occurs at
rate $c_t^N(x)/\eta$, where
$c_t^N(x)=c_{N,\eta,\beta}(\sigma^N_{(t/\eta)-},x)$. The spin drift
and predictable brackets are therefore
\begin{align}
 d(\eta^{-1/2}\sigma^N_{t/\eta})
 &=\eta^{-3/2}\bigl[-\sigma^N_{t/\eta}
       +\tanh(\beta\sqrt\eta\,H_t^N)\bigr]dt+dM_t^N,
       \notag\\
 d\langle M^N(x),M^N(y)\rangle_t
 &=\frac4{\eta^2}\mathbf1_{\{x=y\}}c_t^N(x)\,dt.
 \label{eq:stochastic-coordinate-bracket}
\end{align}
Here and below scalar functions act pointwise on $V_N$. Set
\[
 f_\eta(z)=\frac{\tanh(\beta\sqrt\eta\,z)-\sqrt\eta\,z}
                  {\eta^{3/2}},\qquad
 S_t^N=\exp\!\left(\frac t\eta(K_N-I)\right).
\]
Since $K_N$ is nonnegative and has row sums at most one, $S^N$ is a
positive contraction on $\ell^\infty(V_N)$.
Applying $K_N$ to the spin equation gives the exact local-field equation
\begin{equation}\label{eq:stochastic-local-field}
 dH_t^N=\left[\frac{K_N-I}{\eta}H_t^N
                    +K_Nf_\eta(H_t^N)\right]dt+K_N\,dM_t^N.
\end{equation}
Thus its mild form is
\begin{equation}\label{eq:stochastic-local-mild}
 \begin{aligned}
 H_t^N&=S_t^NH_0^N
       +\int_0^tS_{t-s}^NK_Nf_\eta(H_s^N)\,ds+Z_t^N,
 \\
 Z_t^N&=\int_0^tS_{t-s}^NK_N\,dM_s^N.
 \end{aligned}
\end{equation}
This is the local-field formulation used in
\cite[Section~4]{IsingKac-Phi4-2D}. Its comparison with
\eqref{eq:phi4-mild} requires convergence of the semigroups, the drifts,
and the stochastic convolutions. For the last two terms, the estimates
depend on the size of the field. Taylor's formula gives
\[
 f_\eta(z)=\frac{\beta-1}{\eta}z-\frac{\beta^3}{3}z^3+O_R(\eta),
 \qquad |z|\le R.
\]
Thus $f_\eta\to F$ uniformly on each bounded interval, with Lipschitz
constants bounded uniformly in $N$ there. The noise covariance has a
similar dependence on the field: the Glauber rate satisfies
\[
 \left|c_t^N(x)-\frac12\right|
 \le\frac\beta2\sqrt\eta\,|H_{t-}^N(x)|.
\]
When $\|H_{t-}^N\|_\infty\le R$, the tested martingales
$\langle M^N,\varphi\rangle_{m_N}$ and
$\langle M^N,\psi\rangle_{m_N}$, for fixed $\varphi,\psi\in C(K)$,
therefore have bracket density
\[
 \frac4{n\eta^2}\langle c_t^N\varphi,\psi\rangle_{m_N}
 =\frac2{n\eta^2}\langle\varphi,\psi\rangle_{m_N}
   +O_{R,\varphi,\psi}(\sqrt\eta).
\]
Since $n\eta^2\to\theta$ and $m_N\Rightarrow m$, this converges to
$(2/\theta)\langle\varphi,\psi\rangle_m$, the covariation rate of
$\sqrt{2/\theta}\,W_t(\varphi)$ and $\sqrt{2/\theta}\,W_t(\psi)$.

These estimates would suffice if the local field stayed in a fixed ball. 
{However, this is difficult to establish if true (it may not be): the elementary bound $\|H_t^N\|_\infty\le\eta^{-1/2}$ does not even make the rate error above vanish.}
We therefore {avoid pursing such a bound and opt for another strategy. We} first compare globally defined truncated equations that agree with the original equations until their fields reach norm $R$. 
For fixed $R$, their coefficients can be controlled on the whole time interval. Once their convergence is proved, agreement up to exit will bound the microscopic exit probability by that of the global limiting solution. We can then let $R\to\infty$.

Fix $T>0$ and $R>\|\Phi_0\|_\infty$. Let
$\chi_R(z)=(-R)\vee(z\wedge R)$, and define the truncated drifts
\[
 f_{\eta,R}=f_\eta\circ\chi_R,\qquad F_R=F\circ\chi_R.
\]
They agree with the original drifts on $[-R,R]$ and are constant
on each of the two complementary intervals. The local Taylor expansion
and the derivative bound
\[
 |f_\eta'(z)|
 =\left|\frac{\beta-1}{\eta}
       -\frac\beta\eta\tanh^2(\beta\sqrt\eta\,z)\right|
 \le \left|\frac{\beta-1}{\eta}\right|+\beta^3R^2,
 \qquad |z|\le R,
\]
give
\begin{equation}\label{eq:truncated-drift-convergence}
 \|f_{\eta,R}-F_R\|_\infty\longrightarrow0,\qquad
 \sup_N\operatorname{Lip}(f_{\eta,R})<\infty.
\end{equation}
Thus truncation makes the drift approximation uniform on $\mathbb R$
and gives the Lipschitz bound needed to compare the mild equations.

To control the noise covariance on the whole interval, set
\[
 \tau_R^N=\inf\{t\ge0:\|H_t^N\|_\infty\ge R\}.
\]
Before this exit, the rate estimate above applies. Stopping $M^N$
at $\tau_R^N$ would leave the random factor
$\mathbf1_{\{t\le\tau_R^N\}}$ in its bracket. Since the exit
probability is not yet controlled, this would not identify the full
limiting stochastic convolution. Instead, we keep the original noise
through the exit and continue it afterwards with independent noise
having the required bracket.

Following Mourrat and Weber
\cite[Section~6]{IsingKac-Phi4-2D}, extend $M^N$ on an enlarged
probability space. The extension $M^{N,R}$ agrees with $M^N$ up to and
including $\tau_R^N$. After this time, each coordinate is continued
using independent compensated Poisson processes with jump size
$2/\sqrt\eta$ and rate $1/(2\eta)$, independent of the original system.
The continued martingales have no simultaneous jumps at distinct sites,
and their brackets satisfy
\begin{equation}\label{eq:continued-bracket}
 \begin{aligned}
 d\langle M^{N,R}(x),M^{N,R}(y)\rangle_t
 &=\frac4{\eta^2}\mathbf1_{\{x=y\}}c_t^{N,R}(x)\,dt,\\
 c_t^{N,R}(x)&=c_t^N(x)\mathbf1_{\{t\le\tau_R^N\}}
                +\frac12\mathbf1_{\{t>\tau_R^N\}}.
 \end{aligned}
\end{equation}
Before a positive exit, including its left limit, the local field has
norm at most $R$. After the exit the rate is exactly $1/2$, so
\begin{equation}\label{eq:continued-rate-error}
 \sup_{t>0,\,x\in V_N}|c_t^{N,R}(x)-\tfrac12|
 \le\tfrac12\beta R\sqrt\eta\longrightarrow0.
\end{equation}
Define
\[
 Z_t^{N,R}=\int_0^t S_{t-s}^NK_N\,dM_s^{N,R}.
\]
When viewed on $K$, this field and $Z^N$ denote their harmonic
extensions. They agree for $t\le\tau_R^N$.

The truncated microscopic equation is
\begin{equation}\label{eq:truncated-microscopic-equation}
 H_t^{N,R}=S_t^NH_0^N
   +\int_0^t S_{t-s}^NK_Nf_{\eta,R}(H_s^{N,R})\,ds
   +Z_t^{N,R}.
\end{equation}
The bounded Lipschitz drift gives a unique global c\`adl\`ag solution
on $V_N$ by Picard iteration; we use the same notation for its harmonic
extension. Truncation acts on the drift, so the solution may still
leave the ball of radius $R$. It agrees with $H^N$ through $\tau_R^N$:
uniqueness gives agreement before the exit, and the noise has the same
jump at the exit. This also covers $\tau_R^N=0$, when both fields
start at $H_0^N$.

The limiting truncated equation
\[
 \Phi_t^R=P_t\Phi_0+\int_0^tP_{t-s}F_R(\Phi_s^R)\,ds+Z_t
\]
has a unique global continuous solution. Driven by the same $Z$,
it agrees with $\Phi$ until either field reaches norm $R$.

\subsection{Convergence of the stochastic convolution}

We now prove $Z^{N,R}\Longrightarrow Z$ for each fixed $R$.
Semigroup convergence will identify the limiting covariance, while
bounds on the kernels and their increments will give tightness.

By \eqref{eq:discrete-diagonal-bound} and $n\eta^2\to\theta$,
\begin{equation}\label{eq:diagonal-deletion-bound}
 \|D_N\|_{\infty\to\infty}\le C\eta^{2-\alpha},\qquad
 \frac{\|D_N\|_{\infty\to\infty}}\eta\le C\eta^{1-\alpha}
 \longrightarrow0.
\end{equation}
Under the harmonic-extension convention, for every $T>0$ and
$f\in C(K)$,
\begin{equation}\label{eq:microscopic-semigroup-convergence}
 \begin{aligned}
 \sup_{t\le T}\|S_t^N(f|_{V_N})-P_tf\|_\infty
 &\longrightarrow0,\\
 \|K_N(f|_{V_N})-f\|_\infty&\longrightarrow0,
 \end{aligned}
\end{equation}
uniformly for $f$ in compact subsets of $C(K)$. Indeed, let \(J\) be Poisson with mean \(t/\eta\), so that

$$
\widetilde S_t^N f_N=\mathbb E[P_{\eta J}^Nf_N],
\qquad
\mathbb E|\eta J-t|^2=\eta t.
$$
Fix \(0<\delta\le1\). By Chebyshev's inequality,
$$
\sup_{t\le T}\mathbb P(|\eta J-t|>\delta)
\le \frac{\eta T}{\delta^2}.
$$On the complementary event, \(\eta J\in[0,T+1]\). Thus Proposition \ref{prop:gasket-approximation}, contractivity, and harmonic extension give$$
\sup_{t\le T}
\|\widetilde S_t^Nf_N-P_tf\|_\infty
\le
\sup_{s\le T+1}\|P_s^Nf_N-P_sf\|_\infty
+\omega_f(\delta)
+2\|f\|_\infty\frac{\eta T}{\delta^2},
$$where$$
\omega_f(\delta)
:=
\sup_{\substack{s,t\in[0,T+1]\\|s-t|\le\delta}}
\|P_sf-P_tf\|_\infty.
$$
The first term tends to zero by Proposition 1, while \(\omega_f(\delta)\to0\) by strong continuity. Taking \(N\to\infty\) at fixed \(\delta\), and then \(\delta\downarrow0\), proves the desired uniform convergence.
 Duhamel's formula
and \eqref{eq:diagonal-deletion-bound} give
\[
 \sup_{t\le T}\|S_t^N-\widetilde S_t^N\|_{\infty\to\infty}
 \le T\|D_N\|_{\infty\to\infty}/\eta\longrightarrow0.
\]
The assertion for $K_N$ follows by applying
Proposition~\ref{prop:gasket-approximation} at time $\eta$.
Finite nets and contractivity give uniformity on compact families.

For tightness of the stochastic convolutions, we also need bounds on
the kernels and their increments, including at arbitrarily small times.
Let $q_t^N$ be the density of $S_t^NK_N$ with respect to $m_N$:
\[
 S_t^NK_Ng(x)=\frac1n\sum_{y\in V_N}q_t^N(x,y)g(y).
\]
This kernel is symmetric, nonnegative and has mass at most one.
The strong semigroup convergence in \eqref{eq:microscopic-semigroup-convergence} does not itself control \(L^2(m_N)\) kernel increments uniformly at small times. Such estimates are needed down to the smallest graph scales. Moreover, the deleted diagonal \(D_N\) need not commute with \(L_N\), so \(S_t^NK_N\) is not generally diagonal in the graph Laplacian eigenbasis.
In the following estimates, $\|\cdot\|_2$ denotes the norm in
$L^2(m_N)$.

\begin{proposition}\label{prop:microscopic-kernel}
For every $T>0$, there is $C_T$ independent of $N$ such that,
for $0\le t\le T$ and $x,y\in V_N$,
\begin{align}
 \sup_{x,y\in V_N}q_t^N(x,y)
       +\sup_{x\in V_N}\|q_t^N(x,\cdot)\|_2^2
 &\le C_T(t+\eta)^{-\alpha},
       \label{eq:noise-kernel-size}\\
 \int_0^T\|q_u^N(x,\cdot)-q_u^N(y,\cdot)\|_2^2\,du
 &\le C_T R(x,y)^{\alpha/2}.
       \label{eq:noise-kernel-space}
\end{align}
For $x\in V_N$, $s\ge0$, $h>0$ and $s+h\le T$,
\begin{equation}\label{eq:noise-kernel-time}
 \int_0^s\|q_{u+h}^N(x,\cdot)-q_u^N(x,\cdot)\|_2^2\,du
 +\int_0^h\|q_u^N(x,\cdot)\|_2^2\,du
 \le C_T h^{1-\alpha}.
\end{equation}
\end{proposition}
\begin{proof}
The entrywise inequalities $0\le K_N\le P_\eta^N$ and the
exponential series give
\[
 0\le S_t^NK_N\le e^{-t/\eta}
       \sum_{j\ge0}\frac{(t/\eta)^j}{j!}P_{(j+1)\eta}^N.
\]
For $J$ Poisson with mean $a$, splitting at $J=a/2$ and applying
the exponential Markov inequality gives
$\mathbb E(J+1)^{-\alpha}\le C(1+a)^{-\alpha}$.
Together with \eqref{eq:discrete-diagonal-bound}, this proves the
pointwise bound in \eqref{eq:noise-kernel-size}. Nonnegativity and
the mass bound give its $L^2$ part. In particular,
\begin{equation}\label{eq:noise-kernel-integral}
 \sup_x\int_0^h\|q_u^N(x,\cdot)\|_2^2\,du
 \le C_T h^{1-\alpha},\qquad 0<h\le T.
\end{equation}

For the spatial estimate, fix distinct $x,y\in V_N$ and put
$v=n(\mathbf1_{\{x\}}-\mathbf1_{\{y\}})$. All inner products below
are in $L^2(m_N)$. Write
$B_N=(I-P_\eta^N)/\eta$ and $C_N=(I-K_N)/\eta=B_N+D_N/\eta$.
The operator $C_N$ is strictly positive, and $K_N=I-\eta C_N$ commutes
with it. Hence
\begin{equation}\label{eq:noise-kernel-resolvent}
 2\int_0^\infty
       \|q_u^N(x,\cdot)-q_u^N(y,\cdot)\|_2^2\,du
 =\langle v,K_N^2C_N^{-1}v\rangle.
\end{equation}
Let $B_N^\dagger$ denote the inverse of $B_N$ on constants' orthogonal
complement, extended by zero on constants. Since $v\perp1$,
\[
 \langle v,C_N^{-1}v\rangle
 =\sup_g\{2\langle v,g\rangle-\langle g,C_Ng\rangle\}
 \le\langle v,B_N^\dagger v\rangle.
\]
Discarding \(K_N^2\) here would give the bound
\[
\langle v,C_N^{-1}v\rangle
\le \langle v,B_N^\dagger v\rangle
\le R_N(x,y)+2\eta n,
\]
which is insufficient since \(\eta n\to\infty\).
Keeping the factor $K_N^2$ in \eqref{eq:noise-kernel-resolvent} and
using $K_N^2C_N^{-1}=C_N^{-1}-2\eta I+\eta^2C_N$ therefore gives
\[
 \langle v,K_N^2C_N^{-1}v\rangle
 \le\langle v,(B_N^\dagger-\eta I-\eta P_\eta^N)v\rangle
                 +\eta\langle v,D_Nv\rangle.
\]
On a nonconstant eigenvector of $-L_N$ with eigenvalue $\lambda>0$,
the first operator on the right has multiplier
\[
 \frac{\eta}{1-e^{-\eta\lambda}}-\eta-\eta e^{-\eta\lambda}
 =\frac{\eta e^{-2\eta\lambda}}{1-e^{-\eta\lambda}}
 \le\frac1\lambda.
\]
The variational formula for resistance identifies
$\langle v,(-L_N)^\dagger v\rangle=R_N(x,y)$, while
$\langle v,D_Nv\rangle=p_\eta^N(x,x)+p_\eta^N(y,y)$. Thus
\[
 \int_0^\infty\|q_u^N(x,\cdot)-q_u^N(y,\cdot)\|_2^2\,du
 \le\frac12R_N(x,y)
       +\frac\eta2\bigl[p_\eta^N(x,x)+p_\eta^N(y,y)\bigr].
\]
By \eqref{eq:discrete-diagonal-bound}, the second term is
$O(\eta^{1-\alpha})$. The scaling $\eta\asymp3^{-N/2}$ and the
minimum resistance bound of Section~\ref{sec:gasket} give
\[
 \eta^{1-\alpha}\asymp\bigl((3/5)^N\bigr)^{\alpha/2}
       \le C R(x,y)^{\alpha/2}.
\]
Since $R$ is bounded and $R_N=(1+3^{-N})R$, this proves
\eqref{eq:noise-kernel-space}; both sides vanish when $x=y$.

For the temporal estimate, diagonalize $C_N$ and $K_N=I-\eta C_N$.
For every eigenvalue $a>0$,
\[
 \int_0^s(e^{-(u+h)a}-e^{-ua})^2\,du
 =\frac{(1-e^{-ha})^2(1-e^{-2sa})}{2a}
 \le\int_0^h e^{-2ua}\,du.
\]
Summing against the squared coefficients of $K_N(n\mathbf1_{\{x\}})$
in an orthonormal eigenbasis bounds the first integral in
\eqref{eq:noise-kernel-time} by the second. Apply
\eqref{eq:noise-kernel-integral} to conclude.
\end{proof}

The following martingale estimate converts the kernel bounds into moment estimates for the stochastic convolution. The chaining estimate then gives a uniform modulus of continuity on a finite grid in time and space; oscillations between grid times will be controlled separately.

\begin{lemma}\label{lem:auxiliary-martingale}
Fix $N$, $\eta,T>0$ and write $n=|V_N|$. Suppose that
$(M_t(y))_{y\in V_N}$
are compensated pure-jump martingales whose jumps have absolute value
$2/\sqrt\eta$, with predictable jump rates bounded by $1/\eta$ and no
simultaneous jumps at distinct sites. For $p\ge2$ and deterministic
Borel measurable $g:[0,T]\times V_N\to\mathbb R$, set
\[
 Y_t=\frac1n\sum_{y\in V_N}\int_0^t g_s(y)\,dM_s(y).
\]
If the integrals on the right below are finite, then
\begin{align}
 \mathbb E\sup_{t\le T}|Y_t|^p
 \le C_p\bigg[
 &\left(\frac1{n\eta^2}\int_0^T
                 \|g_s\|_{L^2(m_N)}^2\,ds\right)^{p/2}
 \notag\\
 &+\frac1{n^{p-1}\eta^{p/2+1}}\int_0^T
                 \|g_s\|_{L^p(m_N)}^p\,ds\bigg].
 \label{eq:auxiliary-martingale}
\end{align}
The constant depends only on $p$.
\end{lemma}

\begin{proof}
The predictable quadratic variation of $Y$ and the compensator of
$\sum_{s\le T}|\Delta Y_s|^p$ are bounded, respectively, by
\[
 B=\frac4{n\eta^2}\int_0^T\|g_s\|_{L^2(m_N)}^2\,ds,
 \qquad
 J=\frac{2^p}{n^{p-1}\eta^{p/2+1}}
                         \int_0^T\|g_s\|_{L^p(m_N)}^p\,ds.
\]
Indeed, a jump at site $y$ contributes $2|g_s(y)|/(n\sqrt\eta)$
in absolute value. Only the rate bound and absence of simultaneous
jumps enter this computation.

First suppose $g$ is bounded, and stop $Y$ on first reaching absolute
value $L$. The stopped process is bounded by $L$ plus a bounded jump,
and its compensators retain the bounds $B,J$. Apply the jump chain
rule to $|Y|^p$, using
\[
 |u+v|^p-|u|^p-p|u|^{p-2}uv
 \le C_p\bigl(|u|^{p-2}|v|^2+|v|^p\bigr).
\]
The stochastic integral has mean zero. Compensation therefore gives,
with $Y_T^*=\sup_{t\le T}|Y_t|$ for the stopped process,
\[
 \mathbb E|Y_T|^p\le C_pB\,\mathbb E(Y_T^*)^{p-2}+C_pJ.
\]
Doob's inequality and H\"older's inequality imply
\[
 \mathbb E(Y_T^*)^p
 \le C_pB\bigl[\mathbb E(Y_T^*)^p\bigr]^{(p-2)/p}+C_pJ.
\]
For $p>2$, Young's inequality absorbs half of the left side; for $p=2$
the conclusion is immediate. Thus
$\mathbb E(Y_T^*)^p\le C_p(B^{p/2}+J)$, uniformly in $L$.
Fatou's lemma removes the stopping. For general $g$, truncate its
values and apply the same estimate to differences of truncations.
The resulting stochastic integrals converge in the $L^p$ maximal
norm, and their $L^2$ limit is the integral defining $Y$. This proves
\eqref{eq:auxiliary-martingale}.
\end{proof}

\begin{lemma}\label{lem:auxiliary-chaining}
Let $(E,\rho)$ be a finite metric space of diameter at most one.
Suppose that, for every $0<\varepsilon\le1$, it can be covered by at
most $C_0\varepsilon^{-d}$ balls of radius $\varepsilon$ with centres
in $E$, where $C_0,d>0$. Let $Y:E\to\mathbb R$ be a random function satisfying
\[
 \mathbb E|Y(z)-Y(w)|^p\le A\rho(z,w)^p,
 \qquad z,w\in E,
\]
where $p\ge1$ and $A\ge0$. If $0<\kappa<1$ and $p(1-\kappa)>2d$, then
\begin{equation}\label{eq:auxiliary-chaining}
 \mathbb E\sup_{z\ne w}
              \frac{|Y(z)-Y(w)|^p}{\rho(z,w)^{\kappa p}}
 \le C A,
\end{equation}
where $C$ depends only on $p,\kappa,d,C_0$, not on $E$.
The supremum is taken to be zero when $E$ has one point.
\end{lemma}

\begin{proof}
Choose $2^{-k}$-nets $A_k\subset E$ with $|A_k|\le C_0 2^{kd}$,
taking $A_0$ to be a singleton, and choose $\pi_k(z)\in A_k$ with
$\rho(z,\pi_k(z))\le2^{-k}$. Since $E$ is finite, $A_k=E$ and
$\pi_k(z)=z$ for all sufficiently large $k$.

Let $\mathcal L_k$ consist of the pairs in $A_k\times A_{k+1}$ and
$A_k\times A_k$ at distance at most $3\cdot2^{-k}$. Then
$|\mathcal L_k|\le C2^{2kd}$. Define
\[
 \Lambda^p=\sum_{k\ge0}2^{k\kappa p}
                \sum_{(z,w)\in\mathcal L_k}|Y(z)-Y(w)|^p.
\]
Only finitely many summands are nonzero, and the increment bound gives
\[
 \mathbb E\Lambda^p
 \le C A\sum_{k\ge0}2^{-k[p(1-\kappa)-2d]}
 \le C A.
\]
Each retained link at level $k$ has increment at most
$\Lambda2^{-k\kappa}$. For $z\ne w$, choose $k\ge0$ with
$2^{-k-1}<\rho(z,w)\le2^{-k}$. The pair
$(\pi_k(z),\pi_k(w))$ belongs to $\mathcal L_k$, and
$(\pi_j(z),\pi_{j+1}(z))\in\mathcal L_j$ for every $j\ge k$,
as do the corresponding pairs for $w$. Telescoping gives
\[
 |Y(z)-Y(w)|
 \le\Lambda2^{-k\kappa}+2\Lambda\sum_{j\ge k}2^{-j\kappa}
 \le C_\kappa\Lambda\rho(z,w)^\kappa.
\]
Taking the supremum and then the $p$th moment proves
\eqref{eq:auxiliary-chaining}. 
\end{proof}

\begin{proposition}\label{prop:noise-convergence}
For every $T>0$ and $R>\|\Phi_0\|_\infty$,
\[
 Z^{N,R}\Longrightarrow Z\qquad\text{in }D([0,T];C(K)),
\]
where $Z$ is defined in \eqref{eq:limiting-stochastic-convolution}.
\end{proposition}

\begin{proof}
\emph{Tightness.} We first control increments on a time grid, then
show that the oscillations between grid times vanish.
For the kernels and kernel differences below,
$\|g_s\|_\infty\le C_T\eta^{-\alpha}$. Writing
$V=\int_0^T\|g_s\|_2^2\,ds$, Lemma~\ref{lem:auxiliary-martingale} gives
\[
 \mathbb E\left|\frac1n\sum_y\int_0^T g_s(y)\,dM_s^{N,R}(y)\right|^p
 \le C_{p,T}\left[V^{p/2}
       +\eta^{(3/2-\alpha)(p-2)}V\right],\qquad p\ge2,
\]
using $n\asymp\eta^{-2}$ and
$\|g_s\|_p^p\le\|g_s\|_\infty^{p-2}\|g_s\|_2^2$.
For a temporal increment from $s$ to $t$, split the integral at $s$:
the earlier noise carries the kernel difference
$q_{t-u}^N-q_{s-u}^N$, and the later noise carries $q_{t-u}^N$.
For $x,y\in V_N$ and $0\le s<t\le T$,
Proposition~\ref{prop:microscopic-kernel} yields
\begin{align*}
 \mathbb E|Z_t^{N,R}(x)-Z_t^{N,R}(y)|^p
 &\le C_{p,T}R(x,y)^{\alpha p/4},\\
 \mathbb E|Z_t^{N,R}(x)-Z_s^{N,R}(x)|^p
 &\le C_{p,T}\left[|t-s|^{(1-\alpha)p/2}
       +\eta^{(3/2-\alpha)(p-2)}|t-s|^{1-\alpha}\right].
\end{align*}
The spatial jump term is absorbed using
$R(x,y)^{\alpha/4}\ge c\eta^{(1-\alpha)/2}$ for distinct vertices.

Choose $J_N=\lceil T/\eta^2\rceil$ and $h_N=T/J_N\asymp\eta^2$.
This mesh absorbs the temporal jump term and allows the drift
between grid times to be absorbed: $h_N/\eta=O(\eta)$ and
$\eta^{3/2-\alpha}/h_N^{(1-\alpha)/2}=O(\eta^{1/2})$.
Thus, on $\{ih_N:0\le i\le J_N\}\times V_N$,
\[
 \mathbb E|Z_t^{N,R}(x)-Z_s^{N,R}(y)|^p
 \le C_{p,T}\left(|t-s|^{(1-\alpha)/2}
                         +R(x,y)^{\alpha/4}\right)^p.
\]
After rescaling to diameter one, the metric in parentheses has
covering numbers at most $C_T\varepsilon^{-6/(1-\alpha)}$, since
$2/(1-\alpha)+4/(1-\alpha)=6/(1-\alpha)$ by Section~\ref{sec:gasket}.
Take $p=64$. Since
$p(1-\tfrac14)>12/(1-\alpha)$, Lemma~\ref{lem:auxiliary-chaining}
with $\kappa=1/4$ gives random variables $\Lambda_N$ with
$\sup_N\mathbb E\Lambda_N^p<\infty$ such that, on the grid,
\[
 |Z_t^{N,R}(x)-Z_s^{N,R}(y)|
 \le\Lambda_N\left(|t-s|^{(1-\alpha)/8}
                              +R(x,y)^{\alpha/16}\right).
\]
Since $Z_0^{N,R}=0$, the grid supremum is at most $C_T\Lambda_N$.

To control the intervening times, put $Q^N=K_NM^{N,R}$. Then
\[
 Z_t^{N,R}=Q_t^N+\int_0^t\frac{K_N-I}{\eta}Z_s^{N,R}\,ds.
\]
Let $\varepsilon_N$ be the maximum of
$\sup_{a\le u\le a+h_N}|Q_u^N(x)-Q_a^N(x)|$ over all vertices and
grid intervals. Applying Lemma~\ref{lem:auxiliary-martingale} with
$q_0^N$, for which $V\le Ch_N\eta^{-\alpha}$, and summing over the
$nJ_N=O(\eta^{-4})$ site-interval pairs gives
\[
 \mathbb E\varepsilon_N^p
 \le C_{p,T}\left[
   \eta^{(2-\alpha)p/2-4}
   +\eta^{(3/2-\alpha)(p-2)-2-\alpha}\right]\longrightarrow0.
\]
Both exponents are positive for our choice of $p$.
Since $\|(K_N-I)/\eta\|_{\infty\to\infty}\le2/\eta$, the largest
oscillation $\omega_N$ of $Z^{N,R}$ from a left grid endpoint satisfies
\[
 \omega_N\le\varepsilon_N+\frac{2h_N}{\eta}
                         (C_T\Lambda_N+\omega_N).
\]
Absorbing the last term, using $h_N/\eta=O(\eta)$, gives
$\omega_N\le2\varepsilon_N+C_T\eta\Lambda_N\to0$ in $L^p$.

Extend the grid values harmonically in space and linearly in time
to obtain a continuous field $\overline Z^{N,R}$. Contractivity of
harmonic extension gives
\[
 \sup_{t\le T}\|\overline Z_t^{N,R}-Z_t^{N,R}\|_\infty
 \le2\omega_N\xrightarrow{\mathbb P}0.
\]
Linear interpolation preserves the temporal modulus
$C_T\Lambda_N|t-s|^{(1-\alpha)/8}$.
For space, harmonic extension takes convex combinations of cell
boundary values, and boundary vertices of cells containing $x,y$
have resistance distance at most $R(x,y)+C(3/5)^N$. Hence
\[
 |\overline Z_t^{N,R}(x)-\overline Z_t^{N,R}(y)|
 \le C_T\Lambda_N\bigl(R(x,y)+C(3/5)^N\bigr)^{\alpha/16}.
\]
The moment bound on $\Lambda_N$ and $(3/5)^N\to0$ give tight suprema and
asymptotic equicontinuity; continuity for each fixed $N$ handles the
remaining finitely many indices. Arzel\`a--Ascoli gives tightness of
$\overline Z^{N,R}$ in $C([0,T];C(K))$. The uniform approximation above
then gives tightness of $Z^{N,R}$ in $D([0,T];C(K))$, with continuous
limits.

\emph{Identification.} Fix
$0\le t_1,\ldots,t_k\le T$, $\varphi_1,\ldots,\varphi_k\in C(K)$ and
$a_1,\ldots,a_k\in\mathbb R$. By symmetry of $S_t^NK_N$,
\[
 \sum_{i=1}^ka_i\langle Z_{t_i}^{N,R},\varphi_i\rangle_{m_N}
 =Y_N(T),\qquad
 Y_N(t)=\frac1n\sum_y\int_0^t g_N(s,y)\,dM_s^{N,R}(y),
\]
where
\[
 \begin{aligned}
 g_N(s)&=\sum_{i:s\le t_i}a_i
                  S_{t_i-s}^NK_N(\varphi_i|_{V_N}),\\
 g(s)&=\sum_{i:s\le t_i}a_iP_{t_i-s}\varphi_i.
 \end{aligned}
\]
The integrands are uniformly bounded, and
\eqref{eq:microscopic-semigroup-convergence} gives
$\sup_{s\le T}\|g_N(s)-(g(s)|_{V_N})\|_\infty\to0$.
Moreover, the functions $g(s)^2$ lie in a finite union of compact
subsets of $C(K)$. Finite uniform-norm nets therefore make
$m_N\Rightarrow m$ uniform over this family. Together with
\eqref{eq:continued-rate-error} and $n\eta^2\to\theta$, this shows
that the bracket density from \eqref{eq:continued-bracket},
\[
 v_N(s)=\frac4{n\eta^2}
              \langle g_N(s)^2,c_s^{N,R}\rangle_{m_N}
\]
converges to $v(s)=(2/\theta)\|g(s)\|_{L^2(m)}^2$, essentially
uniformly in $s$ with a deterministic error tending to zero. Also
$\sup_s|\Delta Y_N(s)|\le C/(n\sqrt\eta)\to0$, and
$\int_0^T v_N(s)\,ds\le C_T$.

These facts imply a Gaussian limit directly. For $\lambda\in\mathbb R$,
set $f_N(t)=\mathbb E e^{i\lambda Y_N(t)}$. The jump chain rule and
compensation give
\[
 f_N(t)=1-\frac{\lambda^2}{2}\int_0^t
       \mathbb E[e^{i\lambda Y_N(s-)}v_N(s)]\,ds+o(1),
\]
uniformly for $t\le T$. Indeed, the linear term has mean zero, and the
expected absolute Taylor remainder is bounded uniformly in $t\le T$ by
$C_{\lambda,T}/(n\sqrt\eta)\to0$.
Replacing $v_N$ by $v$ yields
\[
 f_N(t)=1-\frac{\lambda^2}{2}\int_0^t v(s)f_N(s)\,ds+o(1).
\]
Gronwall's inequality gives
$f_N(T)\to\exp[-(\lambda^2/2)\int_0^T v(s)\,ds]$.
Consequently the tested values converge jointly to centred Gaussian
variables with the covariance of
$\langle Z_{t_i},\varphi_i\rangle_m$.

Tightness of the fixed-time marginals and uniform convergence of the
$m_N$-integrals on compact subsets of $C(K)$ show that replacing $m_N$
by $m$ in each tested value changes it by a term tending to zero in
probability. The continuous mapping theorem passes these tests to any
continuous subsequential limit. Since $m$ has full support, tests against a countable dense subset
of $C(K)$ at rational times determine a continuous path law. Every
subsequential limit therefore has the law of $Z$.
\end{proof}

\subsection{Finishing the Proof}

\begin{proof}[Proof of Theorem~\ref{thm:stochastic-limit}]
Fix $T>0$. By \eqref{eq:diagonal-deletion-bound},
\[
 \max_{x\in V_N}\left|H_0^N(x)
   -\mathcal X_0^{N,\eta}(p_\eta^N(x,\cdot))\right|
 \le\eta^{-1/2}\|D_N\|_{\infty\to\infty}
 \le C\eta^{3/2-\alpha}\longrightarrow0.
\]
Harmonic extension preserves this bound, so the initial assumption
gives $H_0^N\Longrightarrow\Phi_0$ in $C(K)$.

\emph{Convergence at a fixed truncation level.}
Fix $R>\|\Phi_0\|_\infty$. We first prove
\begin{equation}\label{eq:truncated-field-convergence}
 H^{N,R}\Longrightarrow\Phi^R
 \quad\text{in }D([0,T];C(K)).
\end{equation}
Since $\Phi_0$ is deterministic, the prescribed initial convergence
and Proposition~\ref{prop:noise-convergence} give joint convergence of
$(H_0^N,Z^{N,R})$ to $(\Phi_0,Z)$. By the Skorokhod representation
theorem \cite{billingsley1999convergence}, realize these pairs so that
the initial fields converge uniformly and the noises converge uniformly
on $[0,T]\times K$, almost surely; convergence in $J_1$ to a continuous
path implies uniform convergence. Construct the solutions of
\eqref{eq:truncated-microscopic-equation} from these inputs, preserving
their laws.

For each outcome, compare the two equations at the vertices. The
initial terms converge uniformly on $[0,T]$, by contraction and
\eqref{eq:microscopic-semigroup-convergence}. The same consistency
estimate, uniformly on the compact family
$\{F_R(\Phi_s^R):0\le s\le T\}\subset C(K)$, gives
\[
 \sup_{r,s\le T}
 \left\|S_r^NK_N(F_R(\Phi_s^R)|_{V_N})
       -(P_rF_R(\Phi_s^R))|_{V_N}\right\|_{\infty,V_N}
 \longrightarrow0.
\]
Using \eqref{eq:truncated-drift-convergence} and contractivity, we obtain
\[
 e_N(t)\le o(1)+L_R\int_0^t e_N(s)\,ds,
 \qquad
 e_N(t)=\sup_{u\le t}\max_{x\in V_N}
             |H_u^{N,R}(x)-\Phi_u^R(x)|,
\]
where the error is uniform on $[0,T]$ and $L_R$ is independent of $N$.
Gronwall's inequality gives $e_N(T)\to0$. Harmonic extension is a
contraction, and the extensions of restrictions converge uniformly
on the compact family $\{\Phi_t^R:t\le T\}$. Hence
$\sup_{t\le T}\|H_t^{N,R}-\Phi_t^R\|_\infty\to0$ almost surely,
proving \eqref{eq:truncated-field-convergence}.

\emph{Removal of the truncation.}
Couple $\Phi^R$ and the global solution $\Phi$ using the same $Z$.
Pathwise uniqueness gives agreement until either reaches norm $R$;
by continuity, their events of reaching $R$ by time $T$ coincide.
At finite $N$, $H^N$ has finitely many jumps on $[0,T]$, so a
finite exit is attained. Agreement through the exit therefore gives
\[
 \{\tau_R^N\le T\}
 \subseteq\left\{\sup_{t\le T}\|H_t^{N,R}\|_\infty\ge R\right\}.
\]
The supremum norm over $[0,T]$ is continuous for $J_1$, since time
changes preserve the supremum. Thus Portmanteau and
\eqref{eq:truncated-field-convergence} yield
\begin{equation}\label{eq:cutoff-exit-bound}
 \limsup_{N\to\infty}\mathbb P(\tau_R^N\le T)
 \le\mathbb P\!\left(\sup_{t\le T}\|\Phi_t\|_\infty\ge R\right)
 \longrightarrow0\qquad(R\to\infty).
\end{equation}
The last limit follows from global continuity, proved in
Proposition~\ref{thm:phi4 wp}.

To combine these comparisons, let $G$ be a bounded continuous functional
on $D([0,T];C(K))$ with $\|G\|_\infty\le1$. Agreement outside the
respective exit events gives
\[
 \begin{aligned}
 \left|\mathbb EG(H^N)-\mathbb EG(\Phi)\right|
 &\le 2\mathbb P(\tau_R^N\le T)\\
 &\quad+\left|\mathbb EG(H^{N,R})-\mathbb EG(\Phi^R)\right|\\
 &\quad+2\mathbb P\!\left(\sup_{t\le T}\|\Phi_t\|_\infty\ge R\right).
 \end{aligned}
\]
The middle term tends to zero for fixed $R$ by
\eqref{eq:truncated-field-convergence}. Using \eqref{eq:cutoff-exit-bound}
for the first term, we obtain
\[
 \limsup_{N\to\infty}\left|\mathbb EG(H^N)-\mathbb EG(\Phi)\right|
 \le4\mathbb P\!\left(\sup_{t\le T}\|\Phi_t\|_\infty\ge R\right)
 \longrightarrow0\qquad(R\to\infty).
\]
This proves $H^N\Longrightarrow\Phi$ on $[0,T]$. Applying this
at every integer $T$, where restriction is continuous at the limiting
paths, proves convergence in the local $J_1$ topology on
$D(\mathbb R_+,C(K))$.

\emph{Empirical observables.}
Fix $\varphi\in D_C(L)$,
with restriction to $V_N$ understood when applying discrete operators.
Symmetry of $K_N$ gives
\[
 \mathcal X_t^{N,\eta_N}(\varphi)-\langle H_t^N,\varphi\rangle_{m_N}
 =\eta^{-1/2}\langle\sigma^N_{t/\eta},(I-K_N)\varphi\rangle_{m_N}.
\]
Since $I-P_\eta^N=-\int_0^\eta P_s^NL_N\,ds$, contractivity and
\eqref{eq:gasket-generator-bound} give
$\|(I-P_\eta^N)\varphi\|_\infty\le\eta\|L\varphi\|_\infty$.
Together with \eqref{eq:diagonal-deletion-bound}, this yields the
deterministic estimate
\begin{equation}\label{eq:empirical-field-comparison}
 \begin{aligned}
 \sup_{t\le T}\left|\mathcal X_t^{N,\eta_N}(\varphi)
                 -\langle H_t^N,\varphi\rangle_{m_N}\right|
 &\le\eta^{-1/2}\|(I-K_N)\varphi\|_\infty\\
 &\le\sqrt\eta\,\|L\varphi\|_\infty
       +C\eta^{3/2-\alpha}\|\varphi\|_\infty
 \longrightarrow0.
 \end{aligned}
\end{equation}

The local-field convergence proved above admits a Skorokhod
representation on which $H^N\to\Phi$ uniformly on $[0,T]\times K$
almost surely. Moreover,
$m_N\Rightarrow m$ uniformly over the compact family
$\{\Phi_t\varphi:0\le t\le T\}\subset C(K)$. Consequently,
\[
 \sup_{t\le T}\left|\langle H_t^N,\varphi\rangle_{m_N}
                   -\int_K\Phi_t\varphi\,dm\right|\longrightarrow0
 \qquad\text{almost surely}.
\]
This argument applies jointly to any finite collection of tests.
The comparison \eqref{eq:empirical-field-comparison} transfers the
resulting path convergence to the empirical observables. Since $T$
is arbitrary and the limit is continuous, the convergence holds in
$D(\mathbb R_+,\mathbb R^k)$, as claimed.
\end{proof}

\bibliographystyle{amsplain}
\bibliography{ref}

\end{document}